\documentclass[graybox,envcountsame,envcountsect]{svmult}

\usepackage{mathptmx}       
\usepackage{helvet}         
\usepackage{courier}        
\usepackage{type1cm}        
\usepackage{makeidx}         
\usepackage{graphicx}        
\usepackage{multicol}        
\usepackage[bottom]{footmisc}

\usepackage{amssymb,amsmath,url}
\input xy
\xyoption{all}
\usepackage[utf8]{inputenc}

\newcommand\Ga{\Gamma}
\newcommand\ga{\gamma}

\newcommand{\ra}{\ensuremath{\rightarrow}}

\newcommand{\hol}{\ensuremath{\mathcal{O}}}

\newcommand{\BB}{\ensuremath{\mathbb{B}}}
\newcommand{\BT}{\ensuremath{\mathbb{T}}}
\newcommand{\CC}{\ensuremath{\mathbb{C}}}
\newcommand{\F}{\ensuremath{\mathbb{F}}}
\newcommand{\FF}{\ensuremath{\mathbb{F}}}
\newcommand{\HH}{\ensuremath{\mathbb{H}}}
\newcommand{\LL}{\ensuremath{\mathbb{L}}}
\newcommand{\NN}{\ensuremath{\mathbb{N}}}
\newcommand{\PP}{\ensuremath{\mathbb{P}}}
\newcommand{\QQ}{\ensuremath{\mathbb{Q}}}
\newcommand{\ZZ}{\ensuremath{\mathbb{Z}}}

\DeclareMathOperator{\Alb}{Alb}
\DeclareMathOperator{\Div}{Div}
\DeclareMathOperator{\divi}{div}
\DeclareMathOperator{\Ext}{Ext}
\DeclareMathOperator{\Fix}{Fix}
\DeclareMathOperator{\Hom}{Hom}
\DeclareMathOperator{\Sym}{Sym}
\DeclareMathOperator{\Tors}{Tors}

\newcounter{saveenumi}

\makeindex             

\begin{document}

\title*{Surfaces of general type with geometric genus zero: a survey}

\author{Ingrid Bauer, Fabrizio Catanese and Roberto Pignatelli\thanks{The present work took place in the realm of the DFG
Forschergruppe 790 "Classification of algebraic
surfaces and compact complex manifolds".}}
\institute{Ingrid Bauer \at Lehrstuhl Mathematik VIII, Mathematisches Institut der
Universit\"at Bayreuth; Universit\"atsstr. 30; D-95447 Bayreuth, Germany \email{Ingrid.Bauer@uni-bayreuth.de}
\and Fabrizio Catanese \at Lehrstuhl Mathematik VIII, Mathematisches Institut der
Universit\"at Bayreuth; Universit\"atsstr. 30; D-95447 Bayreuth, Germany \email{Fabrizio.Catanese@uni-bayreuth.de}
\and Roberto Pignatelli \at Dipartimento di Matematica della Universit\`a di Trento;\\ Via
Sommarive 14; I-38123 Trento (TN), Italy \email{Roberto.Pignatelli@unitn.it}}
%
%

\maketitle

\abstract{In the last years there have been several new constructions of surfaces of 
general type with $p_g=0$, and important progress on their classification. The present paper 
presents the status of the art on surfaces of general type with $p_g=0$, and gives an 
updated list of the existing surfaces, in the case where $K^2= 1,...,7$. 
It also focuses on certain important aspects of this classification.
}

\section*{Introduction}

   It is nowadays
well known that minimal surfaces of general type with $p_g(S) = 0$
have invariants
$p_g(S)  =  q (S)= 0, 1 \leq K_S^2 \leq 9$, hence they yield  a
finite number of irreducible components of the
moduli space of surfaces of general type.

At first glance  this class of surfaces seems rather  narrow, but we
want to report on recent results showing how varied and rich is the
botany of such surfaces, for which a
complete classification is still out of reach.

These surfaces represent for algebraic geometers an almost
prohibitive  test case about the possibility of
extending the fine Enriques   classification of special surfaces to
surfaces of general type.

On the one hand, they are the surfaces of general type which achieve
the minimal value $1$ for the
holomorphic  Euler-Poincar\'e characteristic $\chi(S) := p_g(S) -q(S)
+1$, so a naive (and false) guess is that
they should be ``easier'' to understand than  other surfaces with
higher invariants; on the other hand, there
are pathologies (especially concerning the pluricanonical systems)
or problems (cf. the Bloch conjecture
(\cite{bloch}) asserting that for surfaces with $p_g(S)  =  q (S)= 0$
the group of zero cycles modulo rational
equivalence should be  isomorphic to $\ZZ$),  which  only occur for
surfaces with $p_g = 0$.

Surfaces with $p_g(S)  =  q (S)= 0$ have a very old history, dating back to
1896 (\cite{enr96}, see also \cite{enrMS},
I, page 294, and  \cite{Cast}) when Enriques constructed the
so called Enriques surfaces in order to give a counterexample to the conjecture
of  Max Noether that any such surface should be rational, immediately followed by Castelnuovo
who constructed a surface  with $p_g(S)  =  q (S)= 0$ whose bicanonical pencil is  elliptic.

The first surfaces of general type with $p_g = q = 0$ were
constructed in the 1930' s by   Luigi Campedelli
and by Lucien Godeaux (cf. \cite{Cam},
\cite{god}): in their honour  minimal surfaces  of general type with
$K_S^2 = 1$  are called numerical Godeaux surfaces, and those with
$K_S^2 = 2$ are called numerical  Campedelli surfaces.

In the 1970's there was  a big revival of interest in the
construction of these surfaces and in a possible
attempt to classification.

After rediscoveries of these and other old examples a few new ones
were found through the efforts of several
authors, in particular Rebecca Barlow (\cite{barlow}) found a simply
connected numerical Godeaux surface,
which played a decisive role in the study of the differential
topology of algebraic surfaces and 4-manifolds
(and also in the discovery of K\"ahler Einstein metrics of opposite
sign on the same manifold, see \cite{clb}).

A (relatively short) list of the existing  examples appeared  in the
book \cite{bpv}, (see   \cite{bpv}, VII, $11$
and references therein, and see also
\cite{bhpv} for an updated slightly longer list).

There has been  recently  important progress on the topic, and the
goal of the present paper is to present  the
status of the art  on surfaces of general type with $p_g=0$, of
course focusing only on certain aspects of the
story.

Our article is organized as follows: in the first section we explain
the ``fine'' classification problem for
surfaces of general type with $p_g = q=0$.  Since the solution to
this problem is far from sight we pose some
easier problems which could have a greater chance to be solved in the
near future.

Moreover, we try to give an update on  the current knowledge
concerning  surfaces with $p_g=q=0$.

In the second section, we shortly review several reasons why there
has been  a lot of attention devoted to
surfaces  with geometric genus $p_g$ equal to zero: Bloch's
conjecture, the exceptional behaviour of the
pluricanonical maps and the interesting questions whether there are
surfaces of general type homeomorphic
to Del Pezzo surfaces. It is not possible that a surface of general
type be diffeomorphic to a rational surface.
This follows from Seiberg-Witten theory which  brought a breakthrough
establishing in particular that the
Kodaira dimension is a differentiable invariant of the 4-manifold
underlying an algebraic surface.

Since the first step towards a classification  is always the
construction of as many  examples as possible, we
describe in section three various  construction methods for algebraic
surfaces, showing how they lead to
surfaces of general type  with $p_g =0$. Essentially, there are two
different approaches, one is to take
quotients, by a finite or infinite group,  of known (possibly
non-compact) surfaces, and the other is in a
certain sense the dual one, namely constructing the surfaces as
Galois coverings of known surfaces.

The first approach (i.e., taking quotients)  seems at the moment  to
be  far more successful concerning the
number of examples that have been constructed by this method. On the
other hand, the theory of abelian
coverings seems much more useful to study the deformations of the
constructed  surfaces, i.e., to get hold of
the irreducible, resp. connected components of the corresponding moduli spaces.

In the last  section we review some recent results which have
been obtained by the first two authors,
concerning the connected components of  the moduli spaces
corresponding to Keum-Naie, respectively primary
Burniat surfaces.

\section{ Notation}

For typographical reasons, especially lack of space inside the
tables, we shall use the
following non standard notation for a finite cyclic group of order $m$:
$$\ZZ_m  : = \ZZ / m \ZZ  =  \ZZ / m. $$

\noindent
Furthermore $Q_8$ will denote the quaternion group of order 8,
$$Q_8 : = \{ \pm 1,  \pm i, \pm j,  \pm k \}.$$
\noindent
As usual,  $\mathfrak{S}_n$ is the symmetric group in $n$
letters, $\mathfrak{A}_n$ is the alternating
  subgroup.

\noindent
$D_{p,q,r}$ is the generalized dihedral group admitting the following 
presentation:
$$D_{p,q,r}=\langle x,y|x^p,y^q,xyx^{-1}y^{-r} \rangle,$$ while
$D_n=D_{2,n,-1}$ is the usual dihedral group of order $2n$.

$G(n,m)$ denotes the $m$-th group of order $n$ in the MAGMA database 
of small groups.

Finally, we have  semidirect products $H \rtimes \ZZ_r$; to specify
them, one should indicate the image $\varphi \in Aut(H)$ of the 
standard generator of $\ZZ_r$ in
$Aut(H)$. There is no space in the tables to indicate  $\varphi$, hence
we explain here which automorphism   $\varphi$ will be in the case of the
semidirect products occurring as fundamental groups.

For $H = \ZZ^2$ either $r$ is even, and then  $\varphi$ is $-Id$, or
$r=3$ and  $\varphi$ is the matrix $\begin{pmatrix}
-1&-1\\
1&0
\end{pmatrix}$.

Else $H$ is finite and $r=2$;
for
$H=\ZZ_3^2$,  $\varphi$ is
$-Id$; for $H=\ZZ_2^4$,  $\varphi$ is
$\begin{pmatrix} 1&0\\ 1&1
\end{pmatrix} \oplus \begin{pmatrix} 1&0\\ 1&1
\end{pmatrix}.$

Concerning the case where the group $G$ is a semidirect product, we
  simply refer to
\cite{4names} for more details.

Finally, $\Pi_g$ is the fundamental group of a compact Riemann 
surface of genus $g$.

\section{The classification problem and ``simpler'' sub-problems}

The history of surfaces with geometric genus equal to zero starts
about 120 years ago with a question posed
by Max Noether.

Assume that $S \subset \PP^N_{\CC}$ is a smooth projective surface.
Recall that the {\em geometric genus}
of $S$:
$$ p_g(S) :=h^0(S, \Omega^2_S) := \dim H^0(S, \Omega^2_S),
$$ and the {\em irregularity} of $S$:
$$ q(S) :=h^0(S, \Omega^1_S) := \dim H^0(S, \Omega^1_S),
$$ are {\em birational invariants} of $S$.

Trying to generalize the one dimensional situation, Max Noether asked
the following:

\begin{question}
   Let $S$ be a smooth projective surface with $p_g(S) = q(S) = 0$.
Does this imply that $S$ is rational?
\end{question}

The first negative answer to this question is, as we already wrote,
due to Enriques (\cite{enr96}, see also
\cite{enrMS}, I, page 294) and Castelnuovo, who constructed
counterexamples which are surfaces
of special type (this means, with Kodaira dimension $\leq 1$.
 {\em Enriques surfaces} have Kodaira
dimension equal to $0$, Castelnuovo surfaces have instead  Kodaira dimension $1$).

After the already mentioned examples  by   Luigi Campedelli and by
Lucien Godeaux and the new examples
found by Pol Burniat (\cite{burniat}), and by many other authors,
the  discovery and understanding of
surfaces of general type with $p_g=0$ was considered as a challenging
problem (cf. \cite{dolgachev}): a
complete  fine classification  however soon seemed to be  far out of reach.

   Maybe this was the motivation for D. Mumford to ask the following provocative

\begin{question}[Montreal 1980] Can a computer classify all surfaces
of general type with $p_g = 0$?
\end{question} Before we comment more on Mumford's question, we shall
recall some basic facts concerning
surfaces of general type.

Let $S$ be a {\em minimal} surface of general type, i.e., $S$ does
not contain any rational curve of self
intersection $(-1)$, or equivalently, the canonical divisor $K_S$ of
$S$ is nef and big ($K_S^2 > 0$). Then it
is well known that
$$ K_S^2 \geq 1, \  \chi(S):= 1-q(S) + p_g(S) \geq 1.
$$ In particular, $p_g(S) = 0 \ \implies \ q(S) = 0$. Moreover, we
have a coarse moduli space parametrizing
minimal surfaces of general type with fixed $\chi$ and $K^2$.

\begin{theorem} For each pair of natural numbers $(x,y)$ 
we have  the
Gieseker moduli space $\mathfrak{M}_{(x,y)}^{can}$,
 whose points correspond to the isomorphism classes of
minimal surfaces $S$ of general type
with $\chi(S) = x$ and $K^2_S =y$.

It is  a  quasi projective
scheme  which is  a coarse moduli space for
the canonical models of minimal
surfaces $S$ of general type with  $\chi(S) = x$ and $K^2_S =y$.

\end{theorem}

An upper  bound for $K_S^2$ is given by  the famous 
Bogomolov-Miyaoka-Yau inequality:

\begin{theorem}[\cite{miyaoka1}, \cite{yau},
\cite{yau2}, \cite{miyaoka2}]
Let $S$ be
a smooth surface of general type. Then
$$ K_S^2 \leq 9\chi(S),
$$ and equality holds if and only if the universal covering of $S$ is
the complex ball $\mathbb{B}_2:=\{(z,w)
\in \CC^2 | |z|^2 + |w|^2 <1\}$.
\end{theorem}

As a note for the non experts: Miyaoka proved in the first paper the
general inequality, which Yau only
proved under the assumption of ampleness of the canonical divisor
$K_S$. But Yau showed that if equality
holds, and  $K_S$ is ample, then the universal cover is the ball; in
the second paper Miyaoka showed  that if
equality holds, then necessarily  $K_S$ is ample.

\begin{remark} Classification of surfaces of general type with $p_g =0$
means therefore to "understand" the
nine moduli spaces $\mathfrak{M}_{(1,n)}^{can}$ for $1 \leq n \leq
9$, in particular, the connected
components of each $\mathfrak{M}_{(1,n)}^{can}$ corresponding to
surfaces with $p_g =0$. Here,
understanding means to describe the connected and irreducible
components and their respective dimensions.
\end{remark} Even if this is the "test-case" with the lowest possible
value for the invariant $\chi(S)$  for
surfaces of general type, still nowadays we are quite far from
realistically seeing how  this goal can be
achieved. It is in particular a quite non trivial question, given two
explicit surfaces with the same invariants
$(\chi, K^2)$, to decide whether they are in the same connected
component of the moduli space.

An easy observation, which indeed is quite useful, is the following:

\begin{remark} Assume that $S$, $S'$ are two minimal surfaces of general
type which are in the same connected
component of the moduli space. Then $S$ and $S'$ are orientedly
diffeomorphic through a diffeomorphism
preserving the Chern class of the canonical divisor;  whence $S$ and
$S'$ are homeomorphic, in particular
they have the same (topological) fundamental group.

Thus the fundamental group $\pi_1$ is the simplest invariant which
distinguishes  connected components
of the moduli space $\mathfrak{M}_{(x,y)}^{can}$.
\end{remark}

So, it seems natural to pose the following questions which sound
"easier" to solve than the complete
classification of surfaces with geometric genus zero.

\begin{question} What are the topological fundamental groups of
surfaces of general type with $p_g = 0$
and $K_S^2 = y$?
\end{question}
\begin{question} Is $\pi_1(S) =: \Gamma$ residually finite, i.e., is
the natural homomorphism
$\Gamma \rightarrow \hat{\Gamma} = \lim_{H \triangleleft_f
\Gamma} (\Gamma/H)$ from
$\Gamma$ to its profinite completion $\hat{\Gamma}$ injective?
\end{question}
\begin{remark} 1) Note that in general fundamental groups of algebraic
surfaces are not residually finite, but all
known examples have $p_g >0$ (cf. \cite{Toledo},
\cite{trento}).

2) There are examples of surfaces $S$, $S'$ with non isomorphic
topological fundamental groups, but whose
profinite completions are isomorphic (cf. \cite{serre}, \cite{galois}).
\end{remark}

\begin{question}\label{a,b} What are the best possible positive
numbers $a,b$ such that
\begin{itemize}
\item $K_S^2 \leq a$ $\implies$ $|\pi_1(S)| < \infty$,
\item $K_S^2 \geq b$ $\implies$ $|\pi_1(S)| = \infty$?
\end{itemize}
\end{question}

In fact,  by Yau's theorem $K_S^2 = 9$ $\implies$ $|\pi_1(S)| =
\infty$. Moreover by \cite{4names} there
exists a surface $S$ with $K_S^2 =6$ and finite fundamental group, so
$b \geq 7$. On the other hand, there
are surfaces with $K^2 = 4$ and infinite fundamental group (cf.
\cite{keum}, \cite{naie}), whence $a \leq 3$.

Note that all known minimal surfaces of general type $S$ with $p_g =
0$ and $K_S^2 =8$ are uniformized by
the bidisk $\BB_1 \times \BB_1$.

\begin{question} Is the universal covering of $S$ with $K_S^2 = 8$
always $\BB_1 \times \BB_1$?
\end{question} An affirmative answer to the above question would give
a negative answer to the following
question of F. Hirzebruch:

\begin{question}[F. Hirzebruch]\label{hirzebruch} Does there exist a
surface of general type homeomorphic
to $\PP^1 \times \PP^1$?

Or homeomorphic to the blow up $\F_1$ of $\PP^2$ in one point ?\end{question}

In the other direction, for $K_S^2 \leq 2$ it is known that the
profinite completion $\hat{\pi}_1$ is finite.
There is the following result:

\begin{theorem} 1) $K_S^2 = 1$ $\implies$ $\hat{\pi_1} \cong \ZZ_m$ for
$1 \leq m \leq 5$ (cf. \cite{tokyo}).

\noindent 2) $K_S^2 = 2$ $\implies$ $|\hat{\pi_1}| \leq 9$ (cf.
\cite{MilesCamp}, \cite{xiao}).
\end{theorem}
The bounds are sharp in both cases, indeed for the case $K_S^2 = 1$
there are examples
with $\pi_1(S) \cong \ZZ_m$ for  all $1 \leq m \leq 5$ and
there is the following conjecture

\begin{conjecture}[M. Reid]\label{milesconj}
$\mathfrak{M}_{(1,1)}^{can}$ has exactly five irreducible components
corresponding to each choice
$\pi_1(S) \cong \ZZ_m$ for  all $1 \leq m \leq 5$.
\end{conjecture}

This conjecture is known to hold true for $m \geq 3$ (cf. \cite{tokyo}).

One can ask similar questions:

\begin{question}

2)  Does $K_S^2 = 2$, $p_g(S)= 0$ imply that $|\pi_1(S)| \leq 9 $?

\noindent 3) Does $K_S^2 = 3$ (and $p_g(S)= 0$) imply that
$|\pi_1(S)| \leq 16$?
\end{question}

\subsection{Update on surfaces with $p_g = 0$} There has been
recently  important progress on surfaces of
general type with $p_g=0$ and the current situation is as follows:

\medskip
\noindent
\underline{$K_S^2 = 9$}:  these surfaces have the unit  ball in
$\CC^2$ as universal cover, and their
fundamental group is an arithmetic subgroup  $\Ga$ of $ SU (2,1)$.

This case seems to be completely classified through exciting new work
of Prasad and Yeung and of Cartright  and Steger 
(\cite{p-y}, \cite{p-yadd}, \cite{CS}) asserting that  the moduli
space  consists exactly of
100 points, corresponding to 50 pairs of complex
conjugate surfaces (cf. \cite{kk}).

\medskip
\noindent
\underline{$K_S^2 = 8$}: we  posed the  question  whether  in this
case the universal cover must be the bidisk
in $\CC^2$.

Assuming this, a complete classification should be possible.

The classification has already been accomplished  in
\cite{bcg}
for the reducible case where there is a finite \'etale cover which is
isomorphic
to a product of curves. In this case there are exactly 18
irreducible connected components
of the moduli space: in fact, 17 such components are listed in \cite{bcg},
and recently Davide Frapporti (\cite{frap}), while rerunning the classification
program, found one more family whose existence had been excluded
by an incomplete analysis.
There are many examples, due to Kuga and Shavel (\cite{kug},
\cite{shav}) for the irreducible case, which
yield (as in the case $K_S^2 = 9$) rigid surfaces (by results of Jost
and Yau \cite{jostyau}); but a complete
classification of this second case is still missing.

The constructions of minimal surfaces of general type with $p_g=0$
and with $K^2_S \leq 7$ available in the
literature (to the best of the authors' knowledge, and excluding the
recent results of the authors, which will
be described later)  are listed in table \ref{tabknown}.

We proceed to a description, with the aim of putting the recent
developments in proper perspective.

\begin{table}
\caption{Minimal surfaces of general type with
    $p_g=0$ and $K^2 \leq 7$ available in the literature}
\label{tabknown}
\begin{tabular}[ht]{|c|c|c|c|l|}
\hline
$K^2$  & $\pi_1$& $\pi_1^{alg}$ & $H_1$& References \\
\hline
\hline 1  &$\ZZ_5$&$\ZZ_5$ &
$\ZZ_5$&\cite{godold}\cite{tokyo}\cite{miyaokagod}\\
    &$\ZZ_4$&$\ZZ_4$ &
$\ZZ_4$&\cite{tokyo}\cite{op}\cite{barlow2}\cite{naie94}\\
    &?& $\ZZ_3$ &  $\ZZ_3$&\cite{tokyo}\\
    &$\ZZ_2$&$\ZZ_2$ & $\ZZ_2$&\cite{barlow2}\cite{inoue}\cite{kl}\\
    &?&$\ZZ_2$ & $\ZZ_2$&\cite{werner}\cite{wer97}\\
    &$\{1\}$& $\{1\}$&  $\{0\}$ &\cite{barlow}\cite{lp} \\
    & ? & $\{1\}$ &  $\{0\}$& \cite{cg}\cite{dw}  \\
\hline
\hline 2  &$\ZZ_9$&$\ZZ_9$ & $\ZZ_9$& \cite{mlp} \\
     &$\ZZ_3^2$&$\ZZ_3^2$&$\ZZ_3^2$ &\cite{xiao}\cite{mlp} \\

&$\ZZ_2^3$&$\ZZ_2^3$&$\ZZ_2^3$&\cite{Cam}\cite{MilesCamp}\cite{peterscamp}\cite{inoue}\cite{naie94}\\
     &\ $\ZZ_2 \times \ZZ_4$\ &\ \ \ $\ZZ_2 \times \ZZ_4$\ \ \ &\ \  $\ZZ_2 \times \ZZ_4$\  \  &
\cite{MilesCamp}\cite{naie94}\cite{keum} \\
     &$\ZZ_8$&$\ZZ_8$& $\ZZ_8$ & \cite{MilesCamp} \\
     &$Q_8$&$Q_8$ & $\ZZ_2^2$ &\cite{MilesCamp} \cite{beauville96}\\
     &$\ZZ_7$&$\ZZ_7$& $\ZZ_7$ & \cite{cvg} \\
     &?&$\ZZ_6$& $\ZZ_6$ & \cite{np} \\
     &$\ZZ_5$&$\ZZ_5$ & $\ZZ_5$ & \cite{Babbage}\cite{sup} \\
&$\ZZ_2^2$&$\ZZ_2^2$ & $\ZZ_2^2$ &
\cite{inoue}\cite{keum} \\
    &?&$\ZZ_3$ & $\ZZ_3$ & \cite{lp2} \\
    &$\ZZ_2$&$\ZZ_2$ & $\ZZ_2$ & \cite{kl} \\
    &?&$\ZZ_2$ & $\ZZ_2$ & \cite{lp2} \\
     &$\{1\}$& $\{1\}$ & $\{0\}$ & \cite{lp} \\
\hline
\hline 3  &$\ZZ_2^2 \times \ZZ_4$&$\ZZ_2^2 \times \ZZ_4$ & $\ZZ_2^2
\times \ZZ_4$& \cite{naie94}
\cite{keum} \cite{mlp3} \\
     &$Q_8 \times \ZZ_2$&$Q_8 \times \ZZ_2$ & $\ZZ_2^3$ &
\cite{burniat}\cite{peters} \cite{inoue}\\
    &$\ZZ_{14}$&$\ZZ_{14}$ & $\ZZ_{14}$ & \cite{CS} \\
    &$\ZZ_{13}$&$\ZZ_{13}$ & $\ZZ_{13}$ & \cite{CS} \\
     &$Q_8$&$Q_8$ & $\ZZ_2^2$ & \cite{CS} \\
     &$D_4$&$D_4$ & $\ZZ_2^2$& \cite{CS} \\
  &$\ZZ_2 \times \ZZ_4$&$\ZZ_2 \times \ZZ_4$ &$\ZZ_2 \times \ZZ_4$ & \cite{CS} \\
    &$\ZZ_7$&$\ZZ_7$ & $\ZZ_7$ & \cite{CS} \\
    &$\mathfrak S_3$&$\mathfrak S_3$ & $\ZZ_2$ & \cite{CS} \\
    &$\ZZ_6$&$\ZZ_6$ & $\ZZ_6$ & \cite{CS} \\
    &$\ZZ_2 \times \ZZ_2$&$\ZZ_2 \times \ZZ_2$ &$\ZZ_2 \times \ZZ_2$ & \cite{CS} \\
    &$\ZZ_4$&$\ZZ_4$ & $\ZZ_4$ & \cite{CS} \\
    &$\ZZ_3$&$\ZZ_3$ & $\ZZ_3$ & \cite{CS} \\
    &$\ZZ_2$&$\ZZ_2$ & $\ZZ_2$ & \cite{kl}\cite{CS} \\
     &? &? & $\ZZ_2$ & \cite{pps3b} \\
     &$\{1\}$&$\{1\}$ & $\{0\}$ &\cite{pps3}\cite{CS} \\
\hline
\hline 4  &$1 \rightarrow \ZZ^4 \rightarrow \pi_1 \rightarrow
\ZZ_2^2\rightarrow 1$&$\hat{\pi}_1$& $\ZZ_2^3 \times \ZZ_4$ &
\cite{naie94}\cite{keum}\\
     &$Q_8 \times \ZZ_2^2$ &$Q_8 \times \ZZ_2^2$ & $\ZZ_2^4$
&\cite{burniat}\cite{peters}\cite{inoue}\\
     &$ \ZZ_2$&$ \ZZ_2$ & $ \ZZ_2$ & \cite{park} \\
     &$\{1\}$&$\{1\}$ & $\{0\}$ & \cite{pps4} \\
\hline
\hline 5   &$Q_8 \times \ZZ_2^3$&$Q_8 \times \ZZ_2^3$ & $\ZZ_2^5$ &
\cite{burniat}\cite{peters}\cite{inoue}\\
     &?&?&?& \cite{inoue}\\
\hline
\hline 6   &$1 \rightarrow \ZZ^6 \rightarrow \pi_1 \rightarrow
\ZZ_2^3\rightarrow 1$&$\hat{\pi}_1$& $\ZZ_2^6$ &
\cite{burniat}\cite{peters}\cite{inoue}\\
   &$1 \rightarrow \ZZ^6 \rightarrow \pi_1 \rightarrow
\ZZ_3^3\rightarrow 1$&$\hat{\pi}_1$& $\ZZ_3^3 \subset H_1$ & \cite{kulikov}\\
     &?&?&?& \cite{inoue}\cite{mlp4b}\\
\hline
\hline 7   &\ \ $1 \rightarrow \Pi_3 \times \ZZ^4 \rightarrow \pi_1 \rightarrow
\ZZ_2^3\rightarrow 1$\ \ &$\hat{\pi}_1$&?& \cite{inoue}\cite{mlp01}
\cite{toappear}\\
\hline
\end{tabular}
\end{table}

\medskip
\noindent
\underline{$K_S^2 = 1$}, i.e., {\em numerical Godeaux surfaces}:
recall that by conjecture \ref{milesconj}
the moduli space should have exactly five irreducible connected
components, distinguished by the order of
the fundamental group, which should be cyclic of order at most $5$
(\cite{tokyo} settled the case where the
order of the first homology group is at least 3; \cite{barlow},
\cite{barlow2} and \cite{werner} were the first
to show the occurrence of the two other groups).

\medskip
\noindent
\underline{$K_S^2 = 2$}, i.e., {\em numerical Campedelli surfaces}:
here, it is  known that the order of the
algebraic fundamental group is at most $9$, and the cases of order
$8,9$ have been classified by Mendes
Lopes,
     Pardini and Reid (\cite{mlp}, \cite{mlpr}, \cite{MilesCamp}), who
     showed in particular that the fundamental group equals the algebraic
     fundamental group and cannot be the dihedral group $D_4$ of order
$8$. Naie (\cite{naie}) showed that
the group $D_3$ of order $6$ cannot occur as the fundamental group of
a numerical Campedelli surface. By
the work of  Lee and Park (\cite{lp}), one knows that there exist
simply connected numerical Campedelli
surfaces.

Recently, in \cite{4names}, \cite{bp},  the construction of eight families of
     numerical Campedelli surfaces with fundamental group
$\ZZ_3$ was given. Neves and Papadakis  (\cite{np}) constructed a
numerical Campedelli surface with
algebraic fundamental group $\ZZ_6$, while Lee and Park (\cite{lp2})
constructed one with algebraic
fundamental group $\ZZ_2$, and one with algebraic fundamental group
$\ZZ_3$ was added in the second
version of the same paper. Finally Keum and Lee (\cite{kl}) constructed examples with topological 
fundamental group $\ZZ_2$.

Open conjectures are:

\begin{conjecture} Is the fundamental group $\pi_1(S)$ of a numerical
Campedelli surface finite?
\end{conjecture}
\begin{question}\label{qu1} Does every group of order $\leq 9$ except
$D_4$ and $D_3$ occur as
topological fundamental group (not only as algebraic fundamental group)?
\end{question} The answer to question \ref{qu1} is completely open
for $\ZZ_4$;  for
$\ZZ_6$ one suspects that this fundamental group is
realized by the  Neves-Papadakis surfaces.

Note that the existence of the case where $\pi_1(S) = \ZZ_7$ is shown
in the paper \cite{cvg} (where the
result is not mentioned in the introduction).

\medskip
\noindent
\underline{$K_S^2 = 3$}: here there were two examples of non trivial fundamental groups,
the first one due to Burniat and  Inoue,  the second one to Keum and Naie (\cite{burniat},
\cite{inoue}, \cite{keum}, \cite{naie94}).

It is conjectured that for $p_g(S)= 0, K_S^2 = 3$
the algebraic
fundamental group is finite, and one can ask as in 1) above whether
also $\pi_1(S)$ is
finite. Park, Park and Shin (\cite{pps3}) showed the existence of
simply connected
surfaces, and of surfaces with
torsion $\ZZ_2$ (\cite{pps3b}). More recently  Keum and Lee (\cite{kl}) 
constructed an example with $\pi_1(S) = \ZZ_2$.

Other constructions  were given in \cite{sbc},  together with two more examples
    with $p_g(S)=0, K^2 = 4,5$: these turned out however to be the same
as the Burniat surfaces.

In \cite{bp}, the existence of four new fundamental groups is shown.
Then new fundamental groups were shown to occur by Cartright  
and Steger,
while considering quotients of a fake projective plane by an  
automorphism of order 3.

With this method Cartright  
and Steger produced also other examples with $p_g(S)= 0$, $K_S^2 = 3$,
and trivial fundamental group, or with $\pi_1(S) = \ZZ_2$.

\medskip
\noindent
\underline{$K_S^2 = 4$}: there were known up to now three examples of
fundamental groups, the trivial one
(Park, Park and Shin, \cite{pps4}), a finite one, and an infinite
one. In \cite{4names}, \cite{bp} the existence
of 10 new groups, 6 finite and 4 infinite, is shown: thus minimal
surfaces with $K_S^2 = 4$,
$p_g(S) = q(S)=0$ realize at least 13 distinct topological types. Recently, H. Park constructed one more example in \cite{park} raising the number of topological types to $14$.

\medskip
\noindent
\underline{$K_S^2 = 5,6,7$}: there was known up to now only one
example of a fundamental group for
$K_S^2 = 5, 7$.

Instead for $K_S^2 = 6$, there are the Inoue-Burniat surfaces and an
example due to V.
Kulikov (cf. \cite{kulikov}), which contains
$\ZZ_3^3$ in its torsion group. Like in the case of primary Burniat surfaces
one can see that the fundamental group of the
Kulikov surface fits into an exact sequence $$1 \rightarrow \ZZ^6
\rightarrow \pi_1 \rightarrow
\ZZ_3^3\rightarrow 1.$$

$K_S^2 = 5$ : in \cite{bp} the existence of 7 new groups, four of
which are finite, is shown: thus minimal surfaces
with $K_S^2 = 5$,
$p_g(S) = q(S)=0$ realize at least 8 distinct topological types.

$K_S^2 = 6$ : in \cite{4names} the existence of 6 new groups, three
of which finite, is shown: thus minimal
surfaces with $K_S^2 = 6$,
$p_g(S) = q(S)=0$ realize at least 7 distinct topological types.

$K_S^2 = 7$ : we shall show elsewhere (\cite{toappear}) that these surfaces,
constructed by Inoue in \cite{inoue}, have a fundamental group
fitting into an exact sequence
$$1 \rightarrow \Pi_3 \times \ZZ^4 \rightarrow \pi_1 \rightarrow
\ZZ_2^3\rightarrow 1.$$ 
This motivates the following further question
(cf. question \ref{a,b}).

\begin{question}  Is it true that fundamental groups of surfaces of
general type with $ q=p_g=0$ are finite for
$ K^2_S \leq 3$, and infinite for $ K^2_S \geq 7$?
\end{question}

\section{Other reasons why surfaces with $p_g=0$ have been of
interest in the last 30 years}

\subsection{Bloch's conjecture} Another important problem concerning
surfaces with $p_g = 0$ is related to
the problem of rational equivalence of $0$-cycles.

Recall that, for a nonsingular projective variety $X$, $A_0^i(X)$ is
the group of rational equivalence classes
of zero cycles of degree $i$.

\begin{conjecture} Let $S$ be a smooth surface with $p_g = 0$. Then the
kernel $T(S)$ of the natural morphism
(the so-called {\em Abel-Jacobi map}) $A_0^0(S) \ra \Alb(S)$ is trivial.
\end{conjecture} By a beautiful result of D. Mumford
(\cite{mumfordinfinite}), the kernel of the Abel-Jacobi map is
infinite dimensional for surfaces $S$ with $p_g \neq 0$.

The conjecture has been proven for $\kappa(S) <2$ by Bloch, Kas and
Liebermann (cf. \cite{bkl}). If instead
$S$ is of general type, then $q(S) = 0$, whence Bloch's conjecture
asserts for those surfaces that $A_0(S)
\cong \ZZ$.

Inspite of the efforts of many authors, there are only few cases of
surfaces of general type for which Bloch's
conjecture has been verified (cf. e.g. \cite{inosemik},
\cite{barlowbloch}, \cite{keum}, \cite{voisin}).

Recently S. Kimura introduced the following notion of {\em finite
dimensionality} of motives (\cite{kimura}).
\begin{definition}
   Let $M$ be a motive.

   Then $M$ is {\em evenly finite dimensional} if there is a natural
number $n \geq 1 $ such that $\wedge ^n
M = 0$.

   $M$ is {\em oddly finite dimensional} if there is a natural number
$n \geq 1 $ such that $\Sym ^n M = 0$.

   And, finally, $M$ is {\em finite dimensional} if $M = M^+ \oplus
M^-$, where $M^+$ is evenly finite
dimensional and $M^-$ is oddly finite dimensional.
\end{definition}

Using this notation, he proves the following
\begin{theorem}\label{kimura} 1) The motive of a smooth projective curve
is finite dimensional (\cite{kimura}, cor. 4.4.).

2) The product of finite dimensional motives is finite dimensional
(loc. cit., cor. 5.11.).

3) Let $f \colon M \ra N$ be a surjective morphism of motives, and
assume that $M$ is finite dimensional.
Then $N$ is finite dimensional (loc. cit., prop. 6.9.).

4) Let $S$ be a surface with $p_g = 0$ and suppose that the Chow
motive of $X$ is finite dimensional. Then
$T(S) = 0$ (loc.cit., cor. 7.7.).
\end{theorem}

Using the above results we  obtain
\begin{theorem}
   Let $S$ be the minimal model of a product-quotient surface (i.e.,
birational to $(C_1 \times C_2)/G$, where
$G$ is a finite group acting effectively on a product of two compact
Riemann surfaces of respective genera
$g_i \geq 2$) with $p_g = 0$.

    Then Bloch's conjecture holds for $S$, namely, $A_0(S) \cong \ZZ$.
\end{theorem}

\begin{proof}
Let $S$ be the minimal model of $X = (C_1 \times C_2)/G$. Since $X$
has rational singularities $T(X) = T(S)$.

By thm. \ref{kimura}, 2), 3) we have that the motive of $X$ is
finite dimensional, whence, by 4),  $T(S) = T(X) = 0$.

Since $S$ is of general type we have also $ q(S)=0$, hence $A_0^0(S)
= T(S) = 0$.
\end{proof}

\begin{corollary}
All the surfaces in table \ref{K2>4}, \ref{K2<3}, and all the
surfaces in \cite{bacat}, \cite{bcg} satisfy Bloch's
conjecture.
\end{corollary}

\subsection{Pluricanonical maps}

A further motivation for the study of surfaces with $p_g=0$ comes
from the behavior of the pluricanonical
maps of surfaces of general type.

\begin{definition} The $n$-th pluricanonical map $$\varphi_{n} : =
\varphi_{|nK_S|} \colon S \dashrightarrow
\PP^{P_n-1}$$ is the rational map associated to $H^0(\hol_S(nK_S))$.
\end{definition}

We recall that for a curve of general type $\varphi_{n}$ is an
embedding as soon as $n\geq 3$,  and also for
$n=2$, if the curve is not of genus $2$. The situation in dimension
$2$ is much more complicated. We recall:

\begin{definition} The canonical model of a surface of general type is the
normal surface
$$X : = Proj ( \bigoplus_{n=0}^{\infty} H^0(\hol_S(nK_S))),$$
the projective spectrum of the (finitely generated) canonical ring.

$X$ is obtained from its
minimal model $S$ by  contracting all the curves $C$ with $ K_S \cdot
C = 0$, i.e.,
all the smooth rational curves with self
intersection equal to
$-2$.
\end{definition}

The  $n$-th pluricanonical map $\varphi_{|nK_S|}$ of a surface of
general type  is the composition of the
projection onto its canonical model $X$  with $\psi_n : =
\varphi_{|nK_X|}$. So it suffices to study this last
map.

This was done by Bombieri, whose results were later improved by  the
work of several authors. We summarize
these efforts in the following theorem.

\begin{theorem}[\cite{bom}, \cite{miyaokagod}, \cite{bombcat}, \cite{plurican},
\cite{reider}, \cite{francia},\\ \cite{cc},
\cite{4authors}]
\label{bombieri}

   Let $X$ be the canonical model of a surface of general type. Then
\begin{itemize}
\item[i)] $\varphi_{|nK_X|}$ is an embedding for all $n \geq 5$;
\item[ii)]$\varphi_{|4K_X|}$ is an embedding if $K_X^2 \geq 2$;
\item[iii)] $\varphi_{|3K_X|}$ is a morphism if $K_X^2 \geq 2$ and
              an embedding if $K_X^2 \geq 3$;
\item[iv)] $\varphi_{|nK_X|}$ is birational for all $n \geq 3$ unless
      \subitem a) either $K^2=1$, $p_g=2$, $n=3$ or $4$.

In this case
                $X$ is a hypersurface of degree 10 in the weighted
projective space
$\PP(1,1,2,5)$, a finite double cover of the quadric cone $Y: = \PP(1,1,2)$,
    $\varphi_{|3K_X|}(X)$ is birational to $Y$ and isomorphic to an embedding
of the surface
$\FF_2$ in
                  $\PP^5$, while  $\varphi_{|4K_X|}(X)$ is an 
embedding of $Y$ in
                  $\PP^8$.
      \subitem b) Or $K^2=2$, $p_g=3$, $n=3$ (in this case $X$ is a
double cover of $\PP^2$
branched on a curve of degree 8, and
                  $\varphi_{|3K_X|}(X)$ is the image of the Veronese embedding
                  $\nu_3: \PP^2 \ra\PP^9$).
\item[v)] $\varphi_{|2K_X|}$ is a morphism if $K_X^2 \geq 5$ or if
$p_g \neq 0$.
\item[vi)] If $K_X^2 \geq 10$ then $\varphi_{|2K_X|}$ is birational
if and only if $X$ does not admit a
morphism onto a  curve with general fibre of genus $2$.
\end{itemize}
\end{theorem}

The surfaces with $p_g=0$ arose as the difficult case for the
understanding of the tricanonical map, because,
in the first  version of his  theorem, Bombieri could not determine
whether the  tricanonical and
quadricanonical map of the numerical Godeaux and of the numerical
Campedelli surfaces had to be birational.
This was later proved  in
\cite{miyaokagod},  in \cite{bombcat}, and in \cite{plurican}.

It was already known to Kodaira that a morphism onto a smooth curve
with general fibre of genus $2$ forces
the bicanonical map to factor  through the hyperelliptic involution
of the fibres: this is called the {\it
standard  case} for the nonbirationality of the bicanonical map. Part
vi) of  Theorem \ref{bombieri} shows
that there are finitely many families  of surfaces of general type
with bicanonical map nonbirational which
do not present the standard case.  These interesting families have
been  classified under the hypothesis
$p_g>1$ or $p_g=1$, $q \neq 1$: see \cite{oursurvey} for a more
precise  account on this results.

Again, the surfaces with $p_g=0$ are the most difficult and hence
the most interesting, since there are
"pathologies" which can happen only for surfaces with $p_g=0$.

For example, the bicanonical system of a numerical Godeaux surface is
a pencil, and  therefore maps the
surface onto $\PP^1$, while \cite{XiaoBican} showed  that the bicanonical
map of every other surface of
general type has a two dimensional image.  Moreover, obviously for a
numerical Godeaux
surface $\varphi_{|2K_X|}$ is
not a morphism, thus showing that the condition
$p_g \neq 0$ in the point v) of the Theorem \ref{bombieri} is sharp.

Recently, Pardini and Mendes Lopes (cf. \cite{mlp}) showed that there
are more examples of surfaces whose
bicanonical map is not a morphism,  constructing two families of
numerical Campedelli surfaces whose
bicanonical system has two base points.

What it is known on the degree of the bicanonical map of  surfaces
with $p_g = 0$ can be summarized in the
following
\begin{theorem}[\cite{MPdegree},\cite{MPsurvey}, \cite{mlp}] Let $S$ be
a surface with $p_g=q=0$. Then
\begin{itemize}
\item if $K_S^2=9 \Rightarrow \deg \varphi_{|2K_S|}=1$,
\item if $K_S^2=7,8 \Rightarrow \deg \varphi_{|2K_S|}=1$ or $2$,
\item if $K_S^2=5,6 \Rightarrow \deg \varphi_{|2K_S|}=1$, $2$ or $4$,
\item if $K_S^2=3,4 \Rightarrow \deg \varphi_{|2K_S|}\leq 5$; if moreover
     $\varphi_{|2K_S|}$ is a morphism, then $\deg \varphi_{|2K_S|} =1$,
$2$ or $4$,
\item if $K_S^2=2$ (since the image of the bicanonical map is
$\PP^2$, the bicanonical map is non birational),
then $\deg \varphi_{|2K_S|}\leq 8$. In the known examples it has
degree $6$ (and the bicanonical system has
two base points) or $8$ (and the bicanonical system has no base points).
\end{itemize}
\end{theorem}

\subsection{Differential topology}

The surfaces with $p_g=0$ are very interesting also from the point of
view of differential topology, in
particular in the simply connected case. We recall Freedman's theorem.

\begin{theorem}[\cite{Freedman}] Let $M$ be an oriented, closed, simply
connected topological manifold:
then $M$ is determined (up to homeomorphism) by its intersection form
$$ q \colon H_2(M,\ZZ) \times H_2 (M, \ZZ) \rightarrow \ZZ
$$ and by the Kirby-Siebenmann invariant $\alpha(M) \in \ZZ_2$, which
vanishes if and only if $M \times
S^1$ admits a differentiable structure.
\end{theorem}

If $M$ is a complex surface, the Kirby-Siebenmann invariant
automatically vanishes and therefore the
oriented homeomorphism type of $M$
    is determined by the intersection form.

Combining it with a basic result of Serre on indefinite unimodular
forms, and since by \cite{yau} the only
simply connected compact complex surface  whose  intersection form is
definite is $\PP^2$  one concludes
\begin{corollary} The oriented homeomorphism type of any simply connected
complex surface is determined by
the rank, the index and the parity of the intersection form.
\end{corollary}

This gives a rather easy criterion to decide whether two complex
surfaces are orientedly homeomorphic;
anyway two orientedly homeomorphic complex surfaces are not
necessarily diffeomorphic.

In fact, Dolgachev
surfaces (\cite{dolgachev}, see also \cite[IX.5]{bhpv}) give
examples of infinitely many surfaces which are all
orientedly homeomorphic, but pairwise not diffeomorphic; these are
elliptic surfaces with $p_g=q=0$.

As mentioned, every compact complex  surface homeomorphic to $\PP^2$
is diffeomorphic (in fact,
algebraically isomorphic) to $\PP^2$ (cf. \cite{yau}),  so one can
ask a similar question
(cf. e.g. Hirzebruch's question
\ref{hirzebruch}):  if a surface is
homeomorphic to a rational surface, is it also diffeomorphic to it?

Simply connected surfaces of general type with $p_g=0$ give a
negative answer to this question. Indeed, by
Freedman's theorem each simply connected minimal surface $S$ of
general type with $p_g=0$   is orientedly
homeomorphic to a Del Pezzo surface of degree $K_S^2$. Still these
surfaces are not  diffeomorphic to a Del
Pezzo surface because of the following
\begin{theorem}[\cite{FQ}] Let $S$ be a surface of general type. Then
$S$ is not diffeomorphic to a rational
surface.
\end{theorem}

The first simply connected surface of general type with $p_g=0$ was
constructed by R. Barlow in the 80's, and
more examples  have been constructed recently by Y. Lee, J. Park, H.
Park and D. Shin.
We summarize their results in the following
\begin{theorem}[\cite{barlow}, \cite{lp}, \cite{pps3}, \cite{pps4}]\label{scpg0}
$\forall 1 \leq y \leq 4$ there are minimal simply connected surfaces
of general type with $p_g=0$ and
$K^2=y$.
\end{theorem}

\section{Construction techniques}
As already mentioned, a first step towards a classification  is the
construction of  examples.
Here is a short list of different methods for constructing surfaces
of general type with $p_g=0$.

\subsection{Quotients by a finite (resp. : infinite) group}
\subsubsection{Ball quotients} By the Bogomolov-Miyaoka-Yau theorem,
a surface of general type with $p_g =
0$ is uniformized by the two dimensional complex ball $\BB_2$ if and
only if $K_S^2 =9$. These surfaces are
classically called {\em fake projective planes}, since they have the
same Betti numbers as the projective plane
$\PP^2$.

The first example of a fake projective plane was constructed by
Mumford (cf. \cite{mumford}), and later very
few other examples were given (cf.\cite{ishidakato}, \cite{keumfake}).

Ball quotients $S = \BB_2 / \Gamma$, where $\Gamma \leq PSU(2,1)$ is
a discrete, cocompact, torsionfree
subgroup   are strongly rigid surfaces in view of Mostow's rigidity
theorem (\cite{mostow}).

In particular the moduli
space
$\mathfrak{M}_{(1,9)}$ consists of a finite number of isolated points.

The possibility of  obtaining a complete list of these fake planes
seemed rather unrealistic until a
breakthrough came in 2003: a surprising result by Klingler (cf.
\cite{klingler}) showed that the
cocompact, discrete, torsionfree subgroups
$\Gamma \leq PSU(2,1)$ having minimal Betti numbers, i.e., yielding
fake planes, are indeed arithmetic.

This allowed a complete classification of these surfaces carried out
by Prasad and Yeung,  Steger and Cartright
(\cite{p-y}, \cite{p-yadd}): the moduli space  contains exactly 100 points,
corresponding to 50 pairs of complex conjugate
surfaces.

\subsubsection{Product quotient surfaces}

   In a series of papers the following construction was explored
systematically by the authors with the help of the computer algebra
program MAGMA (cf. \cite{bacat},
\cite{bcg}, \cite{4names}, \cite{bp}).

Let $C_1$, $C_2$ be two
compact curves of respective
genera $g_1, g_2 \geq 2$. Assume further that $G$ is a finite group
acting effectively on $C_1 \times C_2$.

In the case where the action of $G$ is free, the quotient surface is
minimal of general type
and is said to be {\em isogenous to a product} (see
\cite{FabIso}).

   If the action is not free we consider  the minimal resolution of
singularities $S'$ of
the normal surface $X:= (C_1 \times C_2) /G$ and its minimal model $S$.
The aim is to give a complete classification of those $S$ obtained as
above which are of general type and
have $p_g = 0$.

One observes that, if the tangent action of the stabilizers is
contained in $SL(2, \CC)$, then
$X$ has Rational Double Points as singularities and is the canonical
model of a surface of general type.
In this case $S'$ is minimal.

Recall the  definition of an orbifold surface group (here the word
`surface' stands for `Riemann surface'):

\begin{definition} An {\em orbifold surface group} of genus $g'$ and multiplicities
$m_1, \dots m_r \in \NN_{\geq2}$ is the  group  presented as follows:

\begin{multline*}
\mathbb T (g';m_1, \ldots ,m_r) :=
\langle a_1,b_1,\ldots , a_{g'},b_{g'}$,  $c_1,\ldots, c_r |\\
      c_1^{m_1},\ldots ,c_r^{m_r},\prod_{i=1}^{g'} [a_i,b_i] \cdot c_1\cdot
\ldots \cdot c_r \rangle.
\end{multline*}

The sequence $(g';m_1, \dots m_r)$ is called the {\em signature} of
the orbifold surface group.

\end{definition}

Moreover, recall the following  special case of {\em Riemann's 
existence theorem}:

\begin{theorem} A finite group $G$ acts as a group of automorphisms on a
compact Riemann surface $C$ of
genus $g$ if and only if there are natural numbers
$g', m_1, \ldots , m_r$, and an `appropriate'  orbifold homomorphism
$$\varphi \colon
\BT(g';m_1,\ldots, m_r)
\rightarrow G$$  such that the Riemann - Hurwitz relation holds:
$$ 2g - 2 = |G|\left(2g'-2 + \sum_{i=1}^r \left(1 -
\frac{1}{m_i}\right)\right).
$$

   "Appropriate"  means that $\varphi$ is surjective and moreover that
the image $\ga_i \in G$ of a generator $c_i$ has order exactly equal to
$m_i$ (the order of $c_i$ in
$\BT(g';m_1,\ldots, m_r)$).
\end{theorem}

In the above situation $g'$ is the genus of $C':=C/G$. The
$G$-cover $C \rightarrow C'$ is branched in $r$ points $p_1, \ldots ,
p_r$ with branching indices $m_1,
\ldots , m_r$, respectively.

Denote as before $\varphi (c_i)$ by $\ga_i \in G$ the image of $c_i$ under
$\varphi$: then  the set of stabilizers for the action of $G$ on $C$ is the set
$$
\Sigma(\ga_1, \ldots , \ga_r) := \cup_{a \in G} \cup_{i=0}^{max \{m_i\}}
\{a\ga_1^ia^{-1}, \ldots a\ga_r^ia^{-1} \}.
$$

\medskip Assume now that there are two epimorphisms
$$
\varphi_1 \colon \BT(g'_1;m_1,\ldots, m_r) \rightarrow G,
$$
$$
\varphi_2 \colon \BT(g'_2;n_1,\ldots, n_s) \rightarrow G,
$$  determined by two Galois covers $\lambda_i \colon C_i \rightarrow
C'_i$, $i = 1, 2$.

\noindent We will assume in the following that $g(C_1), \ g(C_2) \geq
2$, and we shall consider the diagonal
action of $G$ on $C_1 \times C_2$.

We shall say in this situation that the action of $G$ on $C_1
\times C_2$ is  of {\em unmixed} type (indeed, see \cite{FabIso}, there
is always a subgroup of $G$ of index at most 2 with an action of unmixed type).

\begin{theorem}[\cite{bacat}, \cite{BCG} \cite{4names},\cite{bp}]\label{classiso}

1) Surfaces $S$ isogenous to a product with $p_g(S) = q(S) = 0$ form
17 irreducible connected components
of the moduli space $\mathfrak{M}_{(1,8)}^{can}$.

2) Surfaces with $p_g = 0$, whose canonical model is a singular
quotient  $X:=(C_1 \times C_2)/G$ by an
unmixed action of $G$ form 27 further irreducible families.

3) Minimal surfaces with  $p_g=0$ which are the minimal resolution of
the singularities of $X := C_1 \times
C_2/G$ such that the action is of unmixed type and $X$ does not have
canonical singularities
form exactly further 32
irreducible families.

Moreover, $K^2_S = 8$ if and only if $S$ is isogenous to a product.

\end{theorem}

We summarize the above results in tables \ref{K2>4} and \ref{K2<3}.

\begin{table}
\renewcommand{\arraystretch}{1,3}
\caption{Surfaces isogenous to a product and minimal standard
isotrivial fibrations with $p_g=0$, 
$K^2\geq 4$}\label{K2>4}
\begin{tabular}{|c|c|c|c|c|c|c|c|}
\hline
\ $K^2$\ &\ Sing X\ &$T_1$&$T_2$&$G$&\ N\ &$H_1(S,{\mathbb Z})$&$\pi_1(S)$\\
\hline\hline 8& $\emptyset$ & $2, 5^2$ &$3^4$ &  ${\mathfrak A}_5$
&$1$ &$\ZZ_3^2 \times
\ZZ_{15}$&$1 \rightarrow \Pi_{21} \times \Pi_{4} \rightarrow \pi_1
\rightarrow G \rightarrow 1$\\ 8&
$\emptyset$ & $5^3$ & $2^3,3$ & ${\mathfrak A}_5$ & $1$
&$\ZZ_{10}^2$&$1 \rightarrow \Pi_{6} \times
\Pi_{13} \rightarrow \pi_1 \rightarrow G \rightarrow 1$\\ 8&
$\emptyset$ & $3^2,5$ & $2^5$&
${\mathfrak A}_5$ & $1$ &$\ZZ_2^3 \times \ZZ_6$&$1 \rightarrow
\Pi_{16} \times \Pi_{5} \rightarrow
\pi_1 \rightarrow G \rightarrow 1$\\ 8& $\emptyset$ & $2,4,6$&  $2^6$
&  ${\mathfrak S}_4\times
\ZZ_2$ &$1$ &$\ZZ_2^4 \times \ZZ_4$&\ \ $1 \rightarrow \Pi_{25} \times
\Pi_{3} \rightarrow \pi_1
\rightarrow G \rightarrow 1$\ \  \\ 8& $\emptyset$ & $2^2, 4^2$ &$2^3, 4$
&  ${\rm G}(32,27)$  &$1$
&\ $\ZZ_2^2 \times \ZZ_4 \times \ZZ_8$\ &$1 \rightarrow \Pi_{5} \times
\Pi_{9} \rightarrow \pi_1
\rightarrow G \rightarrow 1$\\ 8& $\emptyset$ & $5^3$& $5^3$ &
$\ZZ_5^2$ &$2$ &$\ZZ_5^2$&$1
\rightarrow \Pi_{6} \times \Pi_{6} \rightarrow \pi_1 \rightarrow G
\rightarrow 1$\\ 8& $\emptyset$ &
$3,4^2$ & $2^6$ &  ${\mathfrak S}_4$ &$1$ &$\ZZ_2^4 \times \ZZ_8$&$1
\rightarrow \Pi_{13} \times
\Pi_{3} \rightarrow \pi_1 \rightarrow G \rightarrow 1$\\ 8&
$\emptyset$ & $2^2,4^2$&  $2^2,4^2$&
${\rm G}(16,3)$ &$1$ &$\ZZ_2^2 \times \ZZ_4 \times \ZZ_8$&$1
\rightarrow \Pi_{5} \times \Pi_{5}
\rightarrow \pi_1 \rightarrow G \rightarrow 1$\\ 8& $\emptyset$ &
$2^3,4$ & $2^6$ &  ${\rm
D}_4\times\ZZ_2$ &$1$ &$\ZZ_2^3 \times \ZZ_4^2$&$1 \rightarrow
\Pi_{9} \times \Pi_{3} \rightarrow
\pi_1 \rightarrow G \rightarrow 1$\\ 8& $\emptyset$ &$2^5$  & $2^5$&
$\ZZ_2^4$ &$1$
&$\ZZ_2^4$&$1 \rightarrow \Pi_{5} \times \Pi_{5} \rightarrow \pi_1
\rightarrow G \rightarrow 1$\\ 8&
$\emptyset$ & $3^4$ &  $3^4$&  $\ZZ_3^2$ & $1$ &$\ZZ_3^4$&$1
\rightarrow \Pi_{4} \times \Pi_{4}
\rightarrow \pi_1 \rightarrow G \rightarrow 1$\\ 8& $\emptyset$ &
$2^5$ &$2^6$ &  $\ZZ_2^3$ &$1$
&$\ZZ_2^6$&$1 \rightarrow \Pi_{3} \times \Pi_{5} \rightarrow \pi_1
\rightarrow G \rightarrow 1$\\ 8&
$\emptyset$ & \multicolumn{2}{c|}{mixed} &\  ${\rm G}(256,3678)$\ &$3$
&&\\ 8& $\emptyset$ &
\multicolumn{2}{c|}{mixed} &  \ ${\rm G}(256,3679)$ \ &$1$ &&\\
8& $\emptyset$ &
\multicolumn{2}{c|}{mixed} &  \ ${\rm G}(64,92)$ \ &$1$ &&\\
\hline\hline
   6&$1/2^2$&$2^3, 4$&$2^4, 4$& ${\mathbb Z}_2 \times D_4$ &1&
${\mathbb Z}_2^2 \times {\mathbb
Z}_4^2$ & $1\rightarrow {\mathbb Z}^2 \times \Pi_2 \rightarrow \pi_1
\rightarrow {\mathbb Z}_2^2
\rightarrow 1$\\
   6&$1/2^2$&$2^4,4$ & $2, 4, 6$ &${\mathbb Z}_2 \times {\mathfrak
S}_4$ &1&   ${\mathbb Z}_2^3
\times {\mathbb Z}_4$ &  $1\rightarrow \Pi_2 \rightarrow \pi_1
\rightarrow  {\mathbb Z}_2
\times{\mathbb Z}_4 \rightarrow 1$   \\
   6&$1/2^2$&$2, 5^2$ & $2, 3^3$ & ${\mathfrak A}_5$&1&${\mathbb Z}_3
\times {\mathbb
Z}_{15}$&${\mathbb Z}^2 \rtimes {\mathbb Z}_{15}$\\
   6&$1/2^2$&\ $2, 4, 10$\ &\ $2, 4, 6$\ &${\mathbb Z}_2 \times {\mathfrak
S}_5$&1&${\mathbb Z}_2 \times
{\mathbb Z}_4$&${\mathfrak S}_3 \times D_{4,5,-1}$\\
   6&$1/2^2$&$2, 7^2$&$3^2, 4$&PSL(2,7)&2&${\mathbb Z}_{21}$ &
${\mathbb Z}_7 \times {\mathfrak
A}_4$\\
   6&$1/2^2$&$2, 5^2$&$3^2, 4$&${\mathfrak A}_6$&2&${\mathbb
Z}_{15}$&${\mathbb Z}_5 \times
{\mathfrak A}_4$\\
\hline\hline
   5&$1/3, 2/3$&$2, 4, 6$&$2^4, 3$&${\mathbb Z}_2 \times {\mathfrak
S}_4$&1&${\mathbb Z}_2^2 \times
{\mathbb Z}_4$&$ 1\rightarrow {\mathbb Z}^2 \rightarrow \pi_1
\rightarrow  D_{2,8,3} \rightarrow 1$\\
   5&$1/3, 2/3$&$2^4, 3$&$3, 4^2$&${\mathfrak S}_4$ & 1 &   ${\mathbb
Z}_2^2 \times {\mathbb
Z}_8$&$ 1 \rightarrow {\mathbb Z}^2 \rightarrow \pi_1 \rightarrow
{\mathbb Z}_8 \rightarrow 1$\\
   5&$1/3, 2/3$&$4^2, 6$&$2^3, 3$&${\mathbb Z}_2 \times {\mathfrak
S}_4$ &1& ${\mathbb Z}_2 \times
{\mathbb Z}_8$& $1 \rightarrow {\mathbb Z}^2 \rightarrow \pi_1
\rightarrow {\mathbb Z}_8 \rightarrow
1$\\
   5&$1/3, 2/3$ & $2, 5, 6$ &$3, 4^2$&${\mathfrak S}_5$ &1&${\mathbb
Z}_8$ &              $
D_{8,5,-1}$              \\
   5&\ $1/3, 2/3$\ & $3, 5^2$ &$2^3, 3$&${\mathfrak A}_5$ &1&${\mathbb
Z}_2 \times {\mathbb Z}_{10}$ &
${\mathbb Z}_5 \times Q_8$               \\
   5&$1/3, 2/3$&$2^3, 3$ &$3, 4^2$ &${\mathbb Z}_2^4 \rtimes {\mathfrak
S}_3$ &1& ${\mathbb Z}_2
\times {\mathbb Z}_8$ & $D_{8,4,3}$               \\
   5&$1/3, 2/3$&$3, 5^2$ &$2^3, 3$ &${\mathfrak A}_5$ &1& ${\mathbb
Z}_2 \times {\mathbb Z}_{10}$ &
${\mathbb Z}_2 \times {\mathbb Z}_{10}$              \\
\hline\hline
   4&$1/2^4$&$2^5$&$2^5$&${\mathbb Z}_2^3$ &1& ${\mathbb Z}_2^3 \times
{\mathbb Z}_4$ &
$1\rightarrow {\mathbb Z}^4 \rightarrow \pi_1 \rightarrow {\mathbb
Z}_2^2 \rightarrow 1$       \\
   4&$1/2^4$&$2^2, 4^2$ &$2^2, 4^2$&${\mathbb Z}_2 \times {\mathbb
Z}_4$ &1& ${\mathbb Z}_2^3
\times {\mathbb Z}_4$ &   $1\rightarrow {\mathbb Z}^4 \rightarrow
\pi_1 \rightarrow {\mathbb Z}_2^2
\rightarrow 1$       \\
   4&$1/2^4$&$2^5$&$2^3, 4$&${\mathbb Z}_2 \times D_4$ &1& ${\mathbb
Z}_2^2 \times {\mathbb
Z}_4$ & $1\rightarrow {\mathbb Z}^2 \rightarrow \pi_1 \rightarrow
{\mathbb Z}_2 \times {\mathbb Z}_4
\rightarrow 1$ \\
   4&$1/2^4$&$3, 6^2$&$2^2, 3^2$&${\mathbb Z}_3 \times {\mathfrak S}_3$
&1&${\mathbb Z}_3^2$
&${\mathbb Z}^2 \rtimes {\mathbb Z}_3$           \\
   4&$1/2^4$&$3, 6^2$&$2, 4, 5$&${\mathfrak S}_5$ &1&${\mathbb Z}_3^2$
&${\mathbb Z}^2 \rtimes
{\mathbb Z}_3$          \\
   4&$1/2^4$&$2^5$&$2, 4, 6$&${{\mathbb Z}_2 \times \mathfrak S}_4$
   &1&${\mathbb Z}_2^3$ & ${\mathbb Z}^2 \rtimes {\mathbb Z}_2$ \\
   4&$1/2^4$&$2^2, 4^2$&$2, 4, 6$&${\mathbb Z}_2 \times {\mathfrak
S}_4$ &1&${\mathbb Z}_2^2 \times
{\mathbb Z}_4$ & ${\mathbb Z}^2 \rtimes {\mathbb Z}_4$\\
   4&$1/2^4$&$2^5$&$3, 4^2$&${\mathfrak S}_4$ &1& ${\mathbb Z}_2^2
\times {\mathbb Z}_4$ &
${\mathbb Z}^2 \rtimes {\mathbb Z}_4$\\
   4&$1/2^4$&$2^3, 4$&$2^3, 4$&${\mathbb Z}_2^4 \rtimes {\mathbb Z}_2$
&1& ${\mathbb Z}_4^2$ &
$G(32, 2)$\\
   4&$1/2^4$&$2, 5^2$&$2^2, 3^2$&${\mathfrak A}_5$ &1& ${\mathbb
Z}_{15}$ & ${\mathbb
Z}_{15}$                  \\
   4&$1/2^4$&$2^2, 3^2$&$2^2, 3^2$& ${\mathbb Z}_3^2 \rtimes Z_2$ &1&${\mathbb
Z}_3^3$ &
${\mathbb Z}_3^3$                  \\
   4&$2/5^2$&$2^3, 5$&$3^2, 5$&${\mathfrak A}_5$ &1&    ${\mathbb Z}_2
\times {\mathbb Z}_6$  &
${\mathbb Z}_2 \times {\mathbb Z}_6$  \\
   4&$2/5^2$&$2, 4, 5$&$4^2, 5$& ${\mathbb Z}_2^4 \rtimes D_5$ &3&
${\mathbb Z}_8$  &
${\mathbb Z}_8$                   \\
   4&$2/5^2$&$2, 4, 5$&$3^2, 5$& ${\mathfrak A}_6$ &1&
${\mathbb Z}_6$  &  ${\mathbb
Z}_6$                   \\
\hline
\end{tabular}
\end{table}
\begin{table}
\caption{Minimal standard isotrivial fibrations with $p_g=0$, $K^2\leq 3$}
\label{K2<3}
\renewcommand{\arraystretch}{1,3}
\begin{tabular}{|c|c|c|c|c|c|c|c|}
\hline
\ $K^2$\ &Sing X&$T_1$&$T_2$&$G$&\ N\ &\ $H_1(S,{\mathbb Z})$\ &$\pi_1(S)$\\
\hline\hline
   3&$1/5, 4/5$ &$2^3, 5$&$3^2, 5$& ${\mathfrak A}_5$ &1&    ${\mathbb
Z}_2 \times {\mathbb Z}_6$
&              $ {\mathbb Z}_2 \times {\mathbb Z}_6$               \\
   3&$1/5, 4/5$ &$2, 4, 5$&$4^2, 5$& ${\mathbb Z}_2^4 \rtimes D_5$ &3&
${\mathbb Z}_8$
&                $ {\mathbb Z}_8 $                 \\
   3&$1/3, 1/2^2, 2/3$ &\ $2^2, 3, 4$\ &\ $2, 4, 6$\ &    ${\mathbb Z}_2
   \times {\mathfrak S}_4$ &1&    ${\mathbb Z}_2 \times {\mathbb Z}_4$
&  \            $ {\mathbb Z}_2 \times
{\mathbb Z}_4$      \         \\
   3&$1/5, 4/5$&$2, 4, 5$&$3^2, 5$&    ${\mathfrak A}_6$ &1&
${\mathbb Z}_6$  &
${\mathbb Z}_6$                   \\
\hline\hline
   2&$1/3^2, 2/3^2$&$2, 6^2$ &$2^2, 3^2$& ${\mathbb Z}_2^3 \rtimes
{\mathbb Z}_3$ &1& ${\mathbb
Z}_2^2$ & $Q_8$                   \\
   2&$1/2^6$&$4^3$ &$4^3$ &${\mathbb Z}_4^2$ &1&${\mathbb Z}_2^3$
&${\mathbb Z}_2^3$
\\
   2&$1/2^6$&$2^3, 4$ &$2^3, 4$ &${\mathbb Z}_2 \times D_4$ &1&
${\mathbb Z}_2 \times {\mathbb
Z}_4$ & $  {\mathbb Z}_2 \times {\mathbb Z}_4$                \\
   2&$1/3^2, 2/3^2$&$2^2, 3^2$&$3, 4^2$&${\mathfrak S}_4$ &1& ${\mathbb
Z}_8$ & ${\mathbb
Z}_8$                   \\
   2&$1/3^2, 2/3^2$&$3^2, 5$ &$3^2, 5$ &${\mathbb Z}_5^2 \rtimes
{\mathbb Z}_3$ &2&         $
{\mathbb Z}_5$ &   $              {\mathbb Z}_5$                  \\
   2&$1/2^6$&$2, 5^2$&$2^3, 3$&${\mathfrak A}_5$ &1&${\mathbb Z}_5$
&${\mathbb Z}_5$
\\
   2&$1/2^6$&$2^3, 4$&$2, 4, 6$&${\mathbb Z}_2 \times {\mathfrak
S}_4$&1&$         {\mathbb Z}_2^2$
&$                {\mathbb Z}_2^2$                  \\
   2&$1/3^2, 2/3^2$&$3^2, 5$ &$2^3, 3$ &${\mathfrak A}_5$ &1&${\mathbb
Z}_2^2$ & ${\mathbb
Z}_2^2$                  \\
   2&$1/2^6$&$2, 3, 7$ &$4^3$ &\ PSL(2,7)\ &2& ${\mathbb Z}_2^2$
&${\mathbb Z}_2^2$                  \\
   2&$1/2^6$&$2, 6^2$&$2^3, 3$&${\mathfrak S}_3 \times {\mathfrak S}_3$
&1&$           {\mathbb Z}_3$
&$                 {\mathbb Z}_3 $                  \\
   2&$1/2^6$&$2, 6^2$&$2, 4, 5$&${\mathfrak S}_5$&1&${\mathbb
Z}_3$&${\mathbb Z}_3$                  \\
   2&$1/4, 1/2^2, 3/4$&$2, 4, 7$&$3^2, 4$&     PSL(2,7) &2& ${\mathbb
Z}_3$ &   $              {\mathbb Z}_3
$                  \\
   2&$1/4, 1/2^2, 3/4$&$2, 4, 5$&$3^2, 4$&     $     {\mathfrak
A}_6$&2&  $ {\mathbb Z}_3$ &      $
{\mathbb Z}_3$                   \\
   2&$1/4, 1/2^2, 3/4$&$2, 4, 6$&$2, 4, 5$&   ${\mathfrak S}_5$ &2&
$ {\mathbb Z}_3$ &
${\mathbb Z}_3 $                  \\
\hline\hline
   1&$1/3, 1/2^4, 2/3$&$2^3, 3$ &$3, 4^2$&${\mathfrak S}_4$
&1&${\mathbb Z}_4$ & ${\mathbb
Z}_4$                   \\
   1&$1/3, 1/2^4, 2/3$&$2, 3, 7$&$3, 4^2$&     PSL(2,7) &1& ${\mathbb
Z}_2$ & ${\mathbb
Z}_2$                   \\
   1&\ $1/3, 1/2^4, 2/3$\ &$2, 4, 6$&$2^3, 3$&${\mathbb Z}_2 \times
{\mathfrak S}_4$ &1&   ${\mathbb
Z}_2$ & ${\mathbb Z}_2$ \\
\hline
\end{tabular}
\end{table}

\begin{remark}\label{vanKampen} 1) Recall that, if a diagonal  action of
$G$ on $C_1
\times C_2$ is not free, then
$G$ has a finite set of fixed points. The quotient surface $X:= (C_1
\times C_2) / G$ has a finite number of singular points.
These can be easily found by looking at the given description of the
stabilizers
for the action of $G$ on each individual curve.

Assume that $x \in X$ is a singular point. Then it is a cyclic
quotient singularity
       of type $\frac{1}{n}(1,a)$ with $g.c.d(a,n) = 1$,
        i.e., $X$ is, locally around $x$, biholomorphic to the quotient
of $\mathbb{C}^2$ by
        the action of a diagonal linear automorphism with eigenvalues
        $\exp(\frac{2\pi i}{n})$, $\exp(\frac{2\pi i a}{n})$.
That $g.c.d(a,n) = 1$ follows since the  tangent representation is
faithful on both factors.

2) We denote by $K_X$ the canonical (Weil) divisor on the normal
surface corresponding to $i_* (
\Omega^2_{X^0})$, $ i\colon X^0 \ra X$ being the inclusion of the
smooth locus of $X$. According to
Mumford we have an intersection product with values in $\QQ$ for Weil
divisors on a normal surface, and in
particular we consider the selfintersection  of the canonical divisor,
\begin{equation}\label{K2} K_X^2 =
\frac{8 (g(C_1) - 1) (g(C_2) - 1)}{|G|}
\in
        \mathbb{Q},
\end{equation}
   which is not necessarily an integer.

$K_X^2$ is however an integer (equal indeed to $K_S^2$) if $X$ has
only RDP's  as singularities.

3) The resolution of a cyclic quotient singularity of type
$\frac{1}{n}(1,a)$ with
$g.c.d(a,n) = 1$ is well known. These singularities are resolved by
the so-called {\em Hirzebruch-Jung
strings}. More precisely, let $\pi \colon S \rightarrow X$ be a
minimal resolution of the singularities  and let
$E = \bigcup_{i=1}^m E_i =
\pi^{-1}(x)$. Then $E_i$ is a smooth rational curve with $E_i^2 =
-b_i$ and $E_i \cdot E_j = 0 $ if
$|i-j| \geq 2 $, while$E_i\cdot E_{i+1} = 1$ for
$i
\in\{1,
\ldots , m-1\}$.

   The $b_i$'s are given by the continued fraction
$$
\frac{n}{a} = b_1 - \frac{1}{b_2 - \frac{1}{b_3 - \ldots}}.
$$

Since the minimal resolution  $S'
\rightarrow X$ of the singularities of $X$ replaces each singular
point by a tree of smooth rational curves, we
have, by van Kampen's theorem, that
$\pi_1(X) = \pi_1(S')= \pi_1(S)$.
\end{remark}

Moreover, we can read off all invariants of $S'$ from the group
theoretical data. For details and explicit
formulae we refer to \cite{bp}.

Among others, we also prove the following lemma:
\begin{lemma}\label{bounded} There exist positive numbers $D$, $M$,
$R$, $B$, which depend explicitly
(and only) on the singularities of $X$ such that:
\begin{enumerate}
\item $\chi(S') = 1$ $\implies$ $K_{S'}^2 = 8 - B$;
\item for the corresponding signatures $(0;m_1, \ldots, m_r)$ and
$(0; n_1, \ldots , n_s)$ of the orbifold
surface groups we have $r,s \leq R$, $\forall \ i \ m_i , n_i \leq M$;
\item $|G| = \frac{K_{S'} +
D}{2(-2+\sum_1^r(1-\frac{1}{m_i}))(-2+\sum_1^s(1-\frac{1}{n_i}))}$.
\end{enumerate}
\end{lemma}

\begin{remark} The above lemma \ref{bounded} implies that there is an
algorithm which computes all
such surfaces
$S'$ with
$p_g = q=0$ and fixed $K_{S'}^2$:

\begin{enumerate}
\item[a)] find all possible configurations (= "baskets") $\mathcal{B}$ of
singularities with $B = 8 - K_{S'}^2$;
\item[b)] for a fixed basket $\mathcal{B}$ find all signatures $(0;m_1,
\ldots , m_r)$ satisfying  $2)$;
\item[c)] for each pair of signatures check all groups $G$ of order given
by  $3)$, whether there are surjective
homomorphisms ${\mathbb T}(0;m_i) \rightarrow G$, ${\mathbb T}(0;n_i)
\rightarrow G$;
\item[d)] check whether the  surfaces $X = (C_1 \times C_2) /G$ thus
obtained have the right singularities.
\end{enumerate}
\end{remark}

Still this is not yet the solution of the problem and there are
still several difficult  problems to be overcome:
\begin{itemize}
\item We have to check whether the groups of a given order admit
certain systems of generators of prescribed
orders, and  satisfying moreover certain further conditions (forced by the
basket of singularities); we encounter
in this way groups
of orders 512, 1024, 1536: there are so many groups of these orders 
that the above investigation is not
feasible for naive computer calculations. Moreover, we have to deal 
with groups of orders $>$ 2000: they are not listed in any database
\item If $X$ is singular, we only get subfamilies, not a whole
irreducible component of the moduli space.
There remains the problem of studying the deformations of the minimal
models $S$ obtained with
the above construction.
\item The algorithm is heavy for $K^2$ small. In \cite{bp} we proved
and implemented much stronger
results on the singularities of $X$ and on the possible signatures,
which allowed us to obtain a complete list
of surfaces with $K_S^2 \geq 1$.
\item We have not yet answered completely the original question. Since, if $X$
does not have canonical singularities,
it may happen that
$K_{S'}^2
\leq 0$ (recall that $S'$ is the minimal resolution of singularities
of $X$, which is not necessarily minimal!).
\end{itemize}

Concerning product quotient surfaces, we have proven (in a much more
general setting, cf. \cite{4names}) a
structure theorem for the fundamental group, which helps us to
explicitly identify the fundamental groups of
the surfaces we constructed. In fact, it is not difficult to obtain a
presentation for these fundamental groups,
but  as usual having a presentation is not sufficient to determine
the group explicitly.

We first need the following
\begin{definition}
      We shall call the fundamental group $\Pi_{g}:=\pi_1(C)$ of a
smooth compact complex curve of genus
$g$ a {\em (genus g) surface group}.
\end{definition} Note that we admit also the ``degenerate cases'' $g = 0, 1$.

\begin{theorem}\label{pi} Let $C_1, \ldots , C_n$ be compact complex
curves of respective genera
$g_i \geq 2$ and let $G$ be a finite group acting faithfully on each
$C_i$ as a group of biholomorphic transformations.

   Let $X=(C_1 \times
\ldots \times C_n) / G$, and denote by $S$ a minimal
desingularisation of $X$. Then the fundamental group
$\pi_1(X) \cong \pi_1(S)$ has a normal subgroup $\mathcal N$ of
finite index which  is isomorphic to the
product of surface groups, i.e., there are natural numbers
$h_1, \ldots , h_n \geq 0$ such that $\mathcal N \cong \Pi_{h_1}
\times \ldots \times
\Pi_{h_n}$.
\end{theorem}

\begin{remark} In the case of dimension $n=2$ there is no loss of
generality in assuming that $G$ acts faithfully
on each
$C_i$ (see \cite{FabIso}). In the general case there will be a group
$G_i$, quotient of
$G$, acting faithfully on $C_i$, hence the strategy has to be
slightly changed in the general case.
The generalization of the above theorem, where the assumption that $G$ acts
faithfully on each factor is removed, has
been proven in \cite{dedieuperroni}.
\end{remark}

We shall now give a short outline of the proof of theorem
\ref{pi} in the case $n=2$ (the case of
arbitrary $n$ is  exactly the same).

We have two appropriate orbifold homomorphisms
$$
\varphi_1 \colon \BT_1 : =  \BT(g'_1;m_1,\ldots, m_r) \rightarrow G,
$$
$$
\varphi_2 \colon  \BT_2 : =  \BT(g'_2;n_1,\ldots, n_s) \rightarrow G.
$$

\noindent

We define the fibre product $\HH : = {\HH}(G;\varphi_1,\varphi_2)$ as
\begin{equation} \HH : = {\HH}(G;\varphi_1,\varphi_2):= \{ \,
(x,y)\in \BT_1\times \BT_2\ |\
\varphi_1(x)=\varphi_2(y)\,\}.
\end{equation}

Then the exact sequence
\begin{equation} 1 \rightarrow \Pi_{g_1} \times \Pi_{g_2} \rightarrow
\BT_1 \times \BT_2
\rightarrow G \times G \rightarrow 1,
\end{equation} where $\Pi_{g_i} := \pi_1(C_i)$, induces an exact sequence

\begin{equation} 1 \rightarrow \Pi_{g_1} \times \Pi_{g_2} \rightarrow
{\HH}(G;\varphi_1,\varphi_2)
\rightarrow G \cong \Delta_G \rightarrow 1.
\end{equation} Here $\Delta_G \subset G \times G$ denotes the
diagonal subgroup.

\begin{definition}
      Let $H$ be a group. Then its {\em torsion subgroup} ${\Tors}(H)$
is the normal subgroup  generated by
all elements of finite order in $H$.
\end{definition}

The first observation is that one can calculate our fundamental
groups via a simple algebraic recipe:
$$\pi_1((C_1 \times C_2)/G)
\cong {\HH}(G;\varphi_1,\varphi_2) / {\Tors}(\HH).$$

\medskip\noindent

The strategy is then the following: using the structure of orbifold
surface groups we construct an exact
sequence
$$ 1 \rightarrow E \rightarrow \mathbb{H} / {\Tors}(\mathbb{H}) \rightarrow
\Psi(\hat{\mathbb{H}}) \rightarrow 1,
$$ where
\begin{itemize}
\item[i)] $E$ is finite,
\item[ii)] $\Psi(\hat{\mathbb{H}})$ is a subgroup of finite index in a product
       of orbifold surface groups.
\end{itemize}

Condition
$ii)$ implies that $\Psi(\hat{\mathbb{H}})$ is residually finite and
``good'' according to the following

\begin{definition}[J.-P. Serre] Let $\mathbb{G}$ be a group, and let
$\tilde{\mathbb{G}}$ be its profinite completion. Then $\mathbb{G}$
is said to be {\em good} iff the
homomorphism of cohomology groups
$$ H^k(\tilde{\mathbb{G}},M) \rightarrow H^k(\mathbb{G},M)
$$ is an isomorphism for all $k \in \NN$ and for all finite
$\mathbb{G}$ - modules $M$.
\end{definition}

Then we use the following result due to F. Grunewald, A.
Jaikin-Zapirain, P. Zalesski.

\begin{theorem} {\bf(\cite{GZ})} Let $G$ be residually finite and good,
and let $\varphi \colon H
\rightarrow G$ be surjective with finite kernel. Then $H$ is residually finite.
\end{theorem}

The above theorem implies that $\mathbb{H} / {\Tors}(\mathbb{H})$ is
residually finite, whence there is a
subgroup $\Gamma \leq \mathbb{H} / \Tors(\mathbb{H})$ of finite index such that
$$
\Gamma \cap E = \{1\}.
$$ Now, $\Psi(\Gamma)$ is a subgroup of $\Psi(\hat{\mathbb{H}})$ of finite
       index, whence of finite index in a product of orbifold surface 
groups, and
$\Psi| \Gamma$ is injective. This easily implies our result.

\begin{remark} Note that theorem \ref{pi} in fact yields a geometric
statement in the case
where the genera of the
surface groups are at least 2. Again, for simplicity, we assume that
$n=2$, and suppose that $\pi_1(S)$
has a normal subgroup $\mathcal{N}$ of finite index isomorphic to
$\Pi_g \times \Pi_{g'}$,
with $g, g' \geq 2$. Then there is an
unramified Galois covering $\hat{S}$ of $S$ such that $\pi_1(\hat{S})
\cong \Pi_g \times \Pi_{g'}$. This
implies (see \cite{FabIso}) that there is a finite morphism $\hat{S}
\rightarrow C \times C'$, where $g(C) = g$,
$g(C') = g'$.

Understanding this morphism can lead  to the understanding of the
irreducible or even of the
connected component of the moduli space containing the isomorphism
class $[S]$ of $S$.
   The method can also work in the case where we only have  $g, g' \geq 1$.
   We shall explain how this method works in section
\ref{keumnaie}.
\end{remark}

We summarize the consequences of theorem \ref{classiso} in terms of
"new" fundamental groups of surfaces
with $p_g =0$, respectively "new" connected components of their moduli space.

\begin{theorem}\label{campedelli} There exist eight families
   of product-quotient  surfaces of unmixed type yielding numerical
Campedelli surfaces (i.e., minimal surfaces
with
$K^2_S = 2, p_g (S)=0$) having fundamental group $ \ZZ / 3$.
\end{theorem}

Our classification also shows the existence  of families of
product-quotient  surfaces yielding numerical
Campedelli surfaces  with fundamental groups $\ZZ/5$ (but numerical
Campedelli surfaces  with
fundamental group $ \ZZ/5$ had already been constructed in
\cite{Babbage}), respectively with
fundamental group $(\ZZ/2)^2$ (but such fundamental  group already
appeared in \cite{inoue}), respectively
with fundamental groups $(\ZZ/2)^3$, $Q_8$, $\ZZ/8$ and $\ZZ/2 \times \ZZ/4$.

\begin{theorem}\label{3} There exist six families  of product-quotient
surfaces yielding  minimal surfaces with
$K^2_S = 3, p_g (S)=0$ realizing four new finite fundamental groups,
$\ZZ/2 \times \ZZ/6$, $\ZZ/8$,
$\ZZ/6$ and $\ZZ/2 \times \ZZ/4$.
\end{theorem}

\begin{theorem}\label{4} There exist sixteen families  of
product-quotient  surfaces yielding  minimal surfaces
with
$K^2_S = 4$, $p_g (S)=0$. Eight of these families realize 6 new finite
fundamental groups, $\ZZ/{15}$,
$G(32,2)$, $(\ZZ/3)^3$, $\ZZ/2 \times \ZZ/6$, $\ZZ/8$, $\ZZ/6$. Eight
of these families  realize 4 new
infinite fundamental groups.
\end{theorem}

\begin{theorem}\label{5} There exist seven families  of product-quotient
surfaces yielding  minimal surfaces with
$K^2_S = 5$, $p_g (S)=0$. Four of these families realize four new
finite fundamental groups, $D_{8,5,-1}$,
$\ZZ/5 \times Q_8$, $D_{8,4,3}$, $\ZZ/2 \times \ZZ/{10}$. Three of
these families  realize three new
infinite fundamental groups.
\end{theorem}

\begin{theorem}\label{6} There exist eight families  of product-quotient
surfaces yielding minimal surfaces with
$K^2_S = 6$, $p_g (S)=0$ and realizing 6 new fundamental groups, three
of them finite and three of them
infinite.
    In particular, there exist minimal surfaces of general type with
$p_g=0$, $K^2=6$ and with finite fundamental group.
\end{theorem}

\subsection{Galois coverings and their deformations}

Another standard method for constructing new algebraic surfaces is to
consider abelian Galois-coverings of
known surfaces.

We shall in the sequel recall the structure theorem on normal finite
$\ZZ_2^r$-coverings, $r\geq1$, of
smooth algebraic surfaces $Y$. In fact (cf. \cite{pardiniabelian}, or
\cite{volumemax} for a more topological
approach) this theory holds more generally for any $G$-covering, with
$G$  a finite abelian group.

Since however we do not want here to dwell too much into the general theory
and,  in most of the applications we consider here only the case
$\ZZ_2^2$ is used, we restrict
ourselves to this more special situation.

We shall denote by $G:=\ZZ_2^r$ the Galois group and by
$G^*:=\Hom(G,\CC^*)$ its dual group of
characters which we identify to $G^*:=\Hom(G,\ZZ/2)$ .

Since $Y$ is smooth  any  finite abelian covering $f \colon X \ra Y$
is flat hence in the eigensheaves  splitting
$$ f_* \hol_X = \bigoplus_{\chi \in G^*}  \mathcal{L}^*_{\chi} =
\hol_Y \oplus \bigoplus_{\chi \in
G^*\setminus \{0\}} \hol_Y(- L_{\chi}).
$$
each rank 1 sheaf  $\mathcal{L}^*_{\chi}$ is invertible and
corresponds to a Cartier divisor $- L_{\chi}$.

For each $\sigma \in G$ let $R_{\sigma} \subset X$ be the divisorial
part of the fixed point set of $\sigma$.
Then one associates to $\sigma$ a divisor $D_{\sigma}$ given by
$f(R_{\sigma}) =  D_{\sigma}$; let
$x_{\sigma}$ be a section such that $\divi(x_{\sigma}) = D_{\sigma}$.

Then the algebra structure on $f_*
\hol_X$ is given by the following (symmetric, bilinear) multiplication maps:
$$
\hol_Y(- L_{\chi}) \otimes \hol_Y(- L_{\eta}) \ra \hol_Y(- L_{\chi + \eta}),
$$ given by the section $x_{\chi, \eta} \in H^0(Y,
\hol_Y(L_{\chi}+L_{\eta} - L_{\chi +\eta}))$, defined by
$$ x_{\chi, \eta}:= \prod_{\chi(\sigma) = \eta(\sigma) = 1} x_{\sigma}.
$$
It is now not difficult in this case to show directly the
associativity of the multiplication
defined  above (cf. \cite{pardini} for the general case of an abelian cover).

In particular, the $G$-covering $f \colon X \ra Y$ is embedded in the
vector bundle $\mathbb{V}:=
\bigoplus _{\chi \in G^*} \mathbb{L}_{\chi}$, where
$\mathbb{L}_{\chi}$ is the geometric line bundle
whose sheaf of sections is $\hol_Y(L_{\chi})$, and is there defined
by the equations:
$$ z_{\chi} z_{\eta} = z_{\chi + \eta}\prod_{\chi(\sigma) =
\eta(\sigma) =1}x_{\sigma}.
$$

Note the special case where $\chi = \eta$, when $\chi+\eta$ is the
trivial character $1$, and $z_1 = 1$. In
particular, let $\chi_1, \ldots , \chi_r$ be a basis of $G^*
=\ZZ_2^r$, and set $z_i :=z_{\chi_i}$. Then we get
the following $r$ equations
\begin{equation}\label{double} z_i^2 = \prod _{\chi_i(\sigma) =1} x_{\sigma}.
\end{equation}

These equations determine the extension of the function fields, hence
one gets $X$ as the normalization of
the Galois covering given by (\ref{double}). The main point however
is that the previous formulae
yield indeed the normalization explicitly under the conditions
   summarized  in the following

\begin{proposition}\label{data} A normal finite $G \cong \ZZ_2^r$-covering
of a smooth variety $Y$ is completely
determined by the datum of

\begin{enumerate}
\item reduced effective divisors $ D_{\sigma}$, $ \forall {\sigma}
\in G$, which have no common
components,
\item divisor classes $ L_1, \dots L_r$, for $\chi_1, \dots \chi_r$
a  basis of $G^*$, such that we  have the
following linear equivalence
\item[$(\#)$] $2 L_i \equiv  \sum _{\chi_i (\sigma)  = 1} D_{\sigma}$.
\end{enumerate} Conversely, given the datum of 1) and 2) such that $\#)$
holds, we obtain a normal scheme
$X$ with a finite $G \cong \ZZ_2^r$-covering $f \colon X \ra Y$.
\end{proposition}

\begin{proof}[Idea of the proof.] It suffices to determine the
divisor classes $L_{\chi}$ for the remaining
elements of
$G^* $. But since any $\chi$ is a sum of basis elements, it suffices
to  exploit the fact that the linear
equivalences
$$  L_{\chi +
\eta} \equiv  L_{\eta}  +  L_{\chi}   -  \sum _{\chi (\sigma) = \eta
(\sigma) = 1} D_{\sigma}
$$ must hold, and apply induction. Since the covering is well defined
as the normalization of the Galois cover
given by (\ref{double}), each $L_{\chi}$ is well defined. Then the
above formulae determine explicitly the
ring structure of $ f_* \hol_X$, hence $X$. Finally, condition 1
implies the normality of the cover.
\end{proof}

A natural question is of course: when is the scheme $X$  a variety?
I.e., $X$ being normal, when is $X$
connected, or, equivalently, irreducible? The  obvious answer is that
$X$ is irreducible if and only if the monodromy homomorphism

$$\mu \colon H_1 (Y \setminus (\cup_{\sigma} D_{\sigma}) ,\ZZ) \ra
G$$ is surjective.

\begin{remark} From the  extension of Riemann's existence theorem due to
Grauert and Remmert
(\cite{grauertremmert}) we know that
$\mu$ determines the covering. It is therefore worthwhile to see how
$\mu$ is related to the datum of 1) and 2).

Write for this purpose the branch locus $D : = \sum_{\sigma}
D_{\sigma}$ as a sum of irreducible
components $ D_i$. To each $D_i$ corresponds a simple geometric loop
$\ga_i$ around $D_i$, and we set
$ \sigma_i : = \mu (\ga_i)$. Then we have that
$D_{\sigma} : =  \sum_{\sigma_i = \sigma } D_i$. For each character
$\chi$, yielding a double covering
associated to the composition $ \chi \circ \mu$, we must find a
divisor class $L_{\chi}$ such that $ 2
L_{\chi} \equiv
\sum_{\chi (\sigma ) = 1} D_{\sigma}$.

Consider the exact sequence
$$ H^{2n-2} (Y, \ZZ) \ra H^{2n-2} (D, \ZZ) = \oplus_i \ZZ [D_i] \ra
H_1  (Y \setminus D, \ZZ)
\ra  H_1 (Y, \ZZ)  \ra 0$$ and the similar one with $\ZZ$ replaced by
$\ZZ_2$. Denote by
$\Delta$ the subgroup image of $ \oplus_i \ZZ_2 [D_i]$. The
restriction of $\mu$ to $\Delta$ is completely
determined by the knowledge of the $\sigma_i$'s, and we have
$$ 0 \ra \Delta \ra  H_1 (Y \setminus D, \ZZ_2)
\ra  H_1 (Y, \ZZ_2)  \ra 0 .$$

Dualizing, we get $$ 0 \ra H^1 (Y, \ZZ_2)   \ra  H^1 (Y \setminus D, \ZZ_2)
\ra  \Hom (\Delta, \ZZ_2)  \ra 0 .$$

The datum of $\chi \circ \mu$,  extending $\chi \circ \mu |
_{\Delta}$ is then seen to correspond to an affine
space over the vector space $H^1 (Y, \ZZ_2)$: and since
$H^1 (Y, \ZZ_2)$ classifies divisor classes of 2-torsion on $Y$, we
infer that the different choices of  $
L_{\chi} $ such that $ 2 L_{\chi} \equiv
\sum_{\chi (\sigma ) = 1} D_{\sigma}$ correspond bijectively to all
the possible
   choices for $\chi \circ \mu$.

Applying this to all characters, we find how $\mu$ determines the
building data.

Observe on the other hand that if $\mu$ is not surjective, then there
is a character $\chi$
vanishing on the image of $\mu$, hence the corresponding double cover is
disconnected.

But the above discussion shows that $\chi \circ \mu$ is trivial iff
this covering is disconnected,
if and only if the corresponding element in $H^1 (Y \setminus D,
\ZZ_2)$ is trivial, or, equivalently, iff the divisor class $L_{\chi}$
is trivial.

\end{remark}

We infer then

\begin{corollary} Use the same notation as in prop. \ref{data}. Then the
scheme $X$ is  irreducible if
$\{ \sigma | D_{\sigma} > 0  \}$ generates $G$.

Or, more generally, if
for each character
$\chi$ the class in $H^1 (Y \setminus D, \ZZ_2)$ corresponding to
$\chi \circ \mu$
is nontrivial, or, equivalently, the divisor class $L_{\chi}$
is nontrivial.
\end{corollary}

\begin{proof} We have seen that if $D_{\sigma} \geq D_i \neq 0$, then
$ \mu (\ga_i)  = \sigma$, whence we
infer that $\mu$ is surjective.
\end{proof}

An important role plays  here once more the concept of {\em natural
deformations}. This concept was
introduced for bidouble covers in \cite{cat1},  definition 2.8, and
extended to the case of abelian covers in
\cite{pardiniabelian}, definition 5.1. The two definitions do not
exactly coincide, because Pardini takes a
much larger parameter space: however, the deformations appearing with
both definitions are
the same. To avoid confusion we
   call Pardini's case the case of {\em extended natural deformations}.

\begin{definition} Let $ f \colon X \ra Y$ be a  finite $G \cong \ZZ_2^r$
covering with
$Y$ smooth and $X$ normal, so that $X$ is embedded in the vector bundle
$\mathbb{V}$ defined above and is  defined by equations
$$z_{\chi}  z_{\eta} = z_{\chi + \eta} \prod _{\chi (\sigma) = \eta
(\sigma) = 1} x_{\sigma}.$$ Let
$\psi_{\sigma, \chi} $ be a section $\psi_{\sigma, \chi} \in H^0
(Y,\hol_Y (D_{\sigma} - L_{\chi}))$, given $
\forall \sigma \in G,
\chi \in G^*.$ To such a collection we associate an {\em extended
natural deformation}, namely, the
subscheme of $\mathbb{V}$ defined by equations
$$z_{\chi} z_{\eta} = z_{\chi + \eta} \prod _{\chi (\sigma) = \eta
(\sigma) = 1}\left( \sum_{\theta}
\psi_{\sigma, \theta } \cdot z_{\theta} \right).$$

We have instead a (restricted) {\em natural deformation} if we
restrict ourselves to the $\theta$'s such that
$\theta (\sigma) = 0$,and we consider  only an equation of the form
$$ z_{\chi}  z_{\eta} = z_{\chi + \eta} \prod _{\chi (\sigma) = \eta
(\sigma) = 1}
\left( \sum_{\theta (\sigma) = 0} \psi_{\sigma, \theta } \cdot z_{\theta}
\right).$$
\end{definition}

One can generalize some results, even removing the assumption of
smoothness of $Y$, if one assumes the $G
\cong \ZZ_2^r$-covering to be {\em  locally simple}, i.e., to enjoy
the property that for each point $y \in Y$
the
$\sigma$'s such that $y \in D_{\sigma}$ are a linearly independent
set. This is a  good notion since (compare
\cite{cat1}, proposition 1.1) if also $X$ is smooth  the covering is
indeed locally simple.

One has for instance the following result (see \cite{man4}, section 3):

\begin{proposition}\label{natdef} Let $ f : X \ra Y$ be a locally simple  $G
\cong \ZZ_2^r$ covering  with $Y$
smooth and $X$ normal. Then we have the exact sequence

$$
      \oplus_{ \chi (\sigma ) = 0} ( H^0 (\hol_{D_{\sigma}} (D_{\sigma}
- L_{\chi}))) \ra
    {\Ext}^1_{\hol_X} (\Omega^1_X, \hol_X) \ra
    {\Ext}^1_{\hol_X} ( f^* \Omega^1_Y, \hol_X)  .$$ In particular,
every small deformation of $X$ is a
natural deformation if

\begin{enumerate}
\item $  H^1 (\hol_Y ( - L_{\chi})) = 0$,
\item $ {\Ext}^1_{\hol_X} ( f^* \Omega^1_Y, \hol_X) = 0.$
\setcounter{saveenumi}{\theenumi}
\end{enumerate} If moreover
\begin{enumerate}
\setcounter{enumi}{\thesaveenumi}
\item $H^0 (\hol_Y (D_{\sigma} - L_{\chi})) = 0$ $ \forall
\sigma \in  G,  \chi \in G^*$,
\end{enumerate} every small deformation of $X$ is again a  $G \cong
\ZZ_2^r$-covering.
\end{proposition}

\begin{proof}[Comments on the proof.]

In the above proposition condition 1) ensures that
    $$ H^0 (\hol_Y (D_{\sigma} - L_{\chi})) \ra  H^0 (\hol_{D_{\sigma}}
(D_{\sigma} - L_{\chi}))$$
is surjective.

Condition 2 and the above exact sequence imply then that the
natural deformations are parametrized by a
smooth manifold and have surjective Kodaira-Spencer map, whence they
induce all the infinitesimal
deformations.
\end{proof}

\begin{remark}
   In the following  section we shall see examples where surfaces
with $p_g=0$ arise as double covers and
as bidouble covers. In fact there are many more surfaces arising this
way, see e.g. \cite{sbc}.
\end{remark}

\section{Keum-Naie surfaces and primary Burniat surfaces}
\label{keumnaie} In the nineties J.H. Keum and D. Naie (cf.
\cite{naie94}, \cite{keum}) constructed a family
of surfaces with $K_S^2 =4$ and $p_g = 0$ as double covers of an
Enriques surface with eight nodes and
calculated their fundamental group.

We want here to describe explicitly  the
moduli space of these surfaces.

\noindent The motivation for this investigation arose as follows: 
consider the following two cases of  table \ref{K2>4}
whose fundamental group has the form
$$\ZZ^4 \hookrightarrow \pi_1
\twoheadrightarrow \ZZ_2^2 \ra 0.$$

These cases yield $2$ families of respective dimensions $2$ and
$4$, which can also  be seen as $\ZZ_4
\times \ZZ_2$, resp.  $\ZZ_2^3$,  coverings of $\PP^1 \times \PP^1$
branched in a divisor of type $(4,4)$,
resp. $(5,5)$, consisting entirely of horizontal and vertical lines. 
It turns out that their fundamental groups are
isomorphic to the fundamental groups of the
surfaces constructed by Keum-Naie.

A straightforward computation shows that our family of dimension $4$
is equal to the family constructed by
Keum, and that both families are subfamilies of the one constructed by Naie.

As a matter of fact each surface of our
family of $\ZZ_2^3$ - coverings of $\PP^1 \times \PP^1$ has $4$
nodes. These nodes can be smoothened
simultaneously in a $5$ - dimensional family of $\ZZ_2^3$ -  Galois
coverings of $\PP^1 \times \PP^1$.

It suffices to take a smoothing of each $D_i$, which before the 
smoothing consisted of
a vertical plus a horizontal line.The
full six dimensional component is obtained then as the family of 
natural deformations of
these  Galois coverings.

It is a standard computation in local deformation theory to show that
the six dimensional family of natural
deformations of smooth $\ZZ_2^3$ -  Galois coverings of $\PP^1 \times
\PP^1$ is an irreducible component
of the moduli space. We will not give the details of this
calculation, since we get a stronger result by another method.

In
fact, the main result of \cite{keumnaie} is the following:
\begin{theorem}\label{main}
   Let $S$ be a smooth complex projective surface which is
homotopically equivalent to a Keum-Naie surface.
Then $S$ is a Keum-Naie surface.

The moduli space of Keum-Naie surfaces is irreducible, unirational
of dimension equal to six. Moreover, the
local moduli space of a Keum-Naie surface is smooth.

\end{theorem}

The proof resorts to a slightly different construction of Keum-Naie 
surfaces. We study a $\ZZ_2^2$-action on the product of two elliptic 
curves $E_1' \times E_2'$. This
action has $16$ fixed points and the
quotient is an $8$-nodal Enriques surface. Constructing
$S$ as a double cover of the Enriques
surface is equivalent to constructing an \'etale $\ZZ_2^2$-covering 
$\hat{S}$ of
$S$, whose existence can be inferred  from the structure of the
fundamental group, and which is obtained as a double cover of $E_1'
\times E_2'$ branched in a $\ZZ_2^2$-invariant divisor of type $(4,4)$.
Because $S = \hat{S }/ \ZZ_2^2$.

The structure of this \'etale $\ZZ_2^2$-covering $\hat{S}$ of $S$
is essentially encoded in the fundamental
group $\pi_1(S)$, which can be described as an affine group $\Gamma \in
\mathbb{A}(2,\CC)$.  The key point is that the double cover 
$\hat{\alpha} : \hat{S} \rightarrow E_1'
\times E_2'$ is the
Albanese map of $\hat{S}$.

Assume now that $S'$ is  homotopically equivalent to a Keum-Naie surface $S$.
Then the corresponding \'etale cover $\hat{S'}$ is homotopically
equivalent to $\hat{S}$. Since we know that the degree of the
Albanese map of $\hat{S}$ is equal to two (by
construction), we can conlude the same for the Albanese map of
$\hat{S'}$ and this allows to deduce that
also $\hat{S'}$ is a double cover of a product of elliptic curves.

A calculation of the invariants of a double cover shows that the 
branch locus is
a $\ZZ_2^2$-invariant divisor of
type $(4,4)$.

We are going to sketch the construction of Keum-Naie surfaces and
the proof of theorem \ref{main} in the
sequel. For details we refer to the original article \cite{keumnaie}.

Let $(E,o)$ be any elliptic curve, with a $G = \ZZ_2^2 = \{0,g_1, g_2,
g_1+g_2 \}$ action given by
$$ g_1(z) := z + \eta, \ \ g_2(z) = -z.
$$

\begin{remark}\label{invdiv}
   Let $\eta \in E$ be a $2$ - torsion point of $E$. Then the divisor
$[o] + [\eta] \in Div^2(E)$ is invariant
under $G$, hence the invertible sheaf $\hol_E([o] + [\eta])$ carries
a natural $G$-linearization.
\end{remark} In particular, $G$ acts on $H^0(E,\hol_E([o] + [\eta]))$,
and for the character eigenspaces, we
have the following:

\begin{lemma}\label{h++} Let $E$ be as above, then:
   $$H^0(E, \hol_E([o] + [\eta])) = H^0(E, \hol_E([o] + [\eta]))^{++}
\oplus H^0(E, \hol_E([o] + [\eta]))^{--}.$$

\noindent I.e., $H^0(E, \hol_E([o] + [\eta]))^{+-} = H^0(E,
\hol_E([o] + [\eta]))^{-+} =0$.
\end{lemma}

\begin{remark} Our notation is self explanatory, e.g. 
$$H^0(E, \hol_E([o] +
[\eta]))^{+-} =  H^0(E, \hol_E([o] +
[\eta]))^{\chi},$$ where $\chi$ is the character of $G$ with
$\chi(g_1) = 1$, $\chi(g_2) = -1$.
\end{remark}

Let now $E_i' := \CC / \Lambda_i$, $i=1,2$, where $\Lambda_i := \ZZ
e_i \oplus \ZZ e_i'$, be two complex
elliptic curves. We consider the affine transformations $\gamma_1, \
\gamma_2 \in \mathbb{A}(2,\CC)$,
defined as follows:
$$
\gamma_1 \begin{pmatrix}
    z_1\\z_2
\end{pmatrix} = \begin{pmatrix}
    z_1 + \frac{e_1}{2}\\- z_2
\end{pmatrix}, \ \ \gamma_2 \begin{pmatrix}
    z_1\\z_2
\end{pmatrix} = \begin{pmatrix}
    - z_1\\ z_2 + \frac{e_2}{2}
\end{pmatrix},
$$

\noindent and let $\Gamma \leq \mathbb{A}(2,\CC)$ be the affine group
generated by $\gamma_1,
\gamma_2$ and by the translations $e_1, e_1', e_2, e_2'$.
\begin{remark} i) $\Gamma$ induces a $G:= \ZZ_2^2$-action on $E_1' \times E_2'$.

\noindent ii) While $\gamma_1, \ \gamma_2$ have no fixed points on
$E_1' \times E_2'$, the involution
$\gamma_1 \gamma_2$ has $16$ fixed points on $E_1' \times E_2'$. It
is easy to see that the quotient
$Y:=(E_1' \times E_2')/G$ is an Enriques surface having $8$ nodes,
with canonical double cover the Kummer
surface $(E_1' \times E_2')/<\gamma_1 \gamma_2>$.
\end{remark}

We  lift the $G$-action on $E_1' \times E_2'$ to an
appropriate ramified double cover $\hat{S}$ such
that $G$ acts freely on $\hat{S}$.

To do this, consider the following geometric line bundle $\mathbb{L}$
on $E_1' \times E_2'$, whose
invertible sheaf of sections is given by:
$$
\hol_{E_1'\times E_2'}(\mathbb{L}) :=p_1^*\hol_{E_1'}([o_1] +
[\frac{e_1}{2}]) \otimes
p_2^*\hol_{E_2'}([o_2] + [\frac{e_2}{2}]),
$$ where $p_i: E_1'\times E_2' \rightarrow E_i'$ is the projection to
the i-th factor.

\medskip By remark \ref{invdiv}, the divisor $[o_i] + [\frac{e_i}{2}]
\in \Div^2(E_i')$ is invariant under $G$.
Therefore, we get a natural $G$-linearization on the two line bundles
$\hol_{E_i'}([o_i] + [\frac{e_i}{2}])$,
whence also on $\mathbb{L}$.

Any two $G$-linearizations of $\mathbb{L}$ differ by a character $\chi
: G \rightarrow \CC^*$. We twist the
above obtained linearization of $\mathbb{L}$ with the character 
$\chi$ such that
$\chi(\gamma_1) = 1$, $\chi(\gamma_2)
= -1$.

\begin{definition} Let
$$f \in H^0(E_1' \times E_2', p_1^*\hol_{E_1'}(2[o_1] +
2[\frac{e_1}{2}]) \otimes p_2^*\hol_{E_2'}(2[o_2]
+ 2[\frac{e_2}{2}]))^G$$ be a $G$ - invariant section of
$\mathbb{L}^{\otimes 2}$ and denote by $w$ a
fibre coordinate of $\mathbb{L}$. Let $\hat{S}$ be the double cover
of $E_1' \times E_2'$ branched in $f$,
i.e., $$
\hat{S} = \{w^2 = f(z_1,z_2) \} \subset \mathbb{L}.
$$ Then $\hat{S}$ is a $G$ - invariant hypersurface in $\mathbb{L}$,
and we have a $G$ - action on
$\hat{S}$.

\noindent We call $S:= \hat{S} / G$ a {\em Keum - Naie surface}, if
\begin{itemize}
   \item $G$ acts freely on $\hat{S}$, and
\item $\{f = 0 \}$ has only {\em non-essential singularities}, i.e.,
$\hat{S}$ has at most rational double points.
\end{itemize}
\end{definition}

\begin{remark} If
$$ f \in H^0(E_1' \times E_2', p_1^*\hol_{E_1'}(2[o_1] +
2[\frac{e_1}{2}]) \otimes p_2^*\hol_{E_2'}(2[o_2]
+ 2[\frac{e_2}{2}]))^G
$$ is such that $\{(z_1,z_2) \in E_1' \times E_2' \ | \ f(z_1, z_2) =
0 \} \cap \Fix(\gamma_1 + \gamma_2) =
\emptyset$, then $G$ acts freely on $\hat{S}$.
\end{remark}

\begin{proposition} Let $S$ be a Keum - Naie surface. Then $S$ is a minimal
surface of general type with
\begin{itemize}
   \item[i)] $K_S^2 = 4$,
\item[ii)] $p_g(S) = q(S) = 0$,
\item[iii)] $\pi_1(S) = \Gamma$.
\end{itemize}
\end{proposition}

i) is obvious, since $K_{\hat{S}}^2 = 16$,

ii) is verified via
standard arguments of  representation theory.

iii) follows since $\pi_1 (\hat{S}) = \pi_1 (E_1' \times E_2') $.
\noindent

Let now $S$ be a smooth complex projective surface with $\pi_1(S) =
\Gamma$. Recall that $\gamma_i^2 =
e_i$ for $i = 1,2$. Therefore $\Gamma = \langle \gamma_1, e_1',
\gamma_2, e_2' \rangle$ and we have the
exact sequence
$$ 1 \rightarrow \ZZ^4 \cong \langle e_1, e_1', e_2, e_2' \rangle
\rightarrow \Gamma \rightarrow \ZZ_2^2
\rightarrow 1,
$$

where $e_i \mapsto \gamma_i^2$.

We set $\Lambda_i':= \ZZ e_i \oplus \ZZ e_i'$, hence $\pi_1 (E_1'
\times E_2') = \Lambda_1' \oplus
\Lambda_2'$. We also have the two lattices $\Lambda_i := \ZZ
\frac{e_i}{2} \oplus \ZZ e_i'$.

\begin{remark}
   1) $\Gamma$ is a group of affine transformations on $\Lambda_1
\oplus \Lambda_2$.

\noindent 2) We have an \'etale double cover $E_i' = \CC / \Lambda_i'
\rightarrow E_i := \CC / \Lambda_i$,
which is the quotient by a semiperiod of $E_i'$.
\end{remark}

$\Gamma$ has two subgroups of index two:
$$
\Gamma_1 := \langle \gamma_1, e_1', e_2, e_2' \rangle, \ \
\Gamma_2:=\langle e_1, e_1', \gamma_2, e_2'
\rangle,
$$ corresponding to two \'etale covers of $S$: $S_i \rightarrow S$,
for $i = 1,2$.

Then one can show:
\begin{lemma} The Albanese variety of $S_i$ is $E_i$. In particular,
$q(S_1) = q(S_2) = 1$.
\end{lemma}

\noindent
Let $\hat{S} \rightarrow S$ be the \'etale $\ZZ_2^2$-covering
associated to $\ZZ^4 \cong \langle e_1, e_1',
e_2, e_2' \rangle \triangleleft \Gamma$. Since $\hat{S} \rightarrow
S_i \rightarrow S$, and $S_i$ maps to
$E_i$ (via the Albanese map), we get a morphism
$$ f:\hat{S} \rightarrow E_1 \times E_2 = \CC/ \Lambda_1 \times \CC /
\Lambda_2.
$$ Then the covering of $E_1 \times E_2$ associated to $\Lambda_1'
\oplus \Lambda_2' \leq \Lambda_1
\oplus \Lambda_2$ is $E_1' \times E_2'$, and
since $  \pi_1(\hat{S}) = \Lambda_1'
\oplus \Lambda_2'$ we see that $f$ factors
through $E_1' \times E_2'$ and that the
Albanese map of $\hat{S}$ is $\hat{\alpha} : \hat{S} \rightarrow E_1'
\times E_2'$.

The proof of the main result follows then from
\begin{proposition}\label{dc} Let $S$ be a smooth complex projective
surface, which is homotopically equivalent to
a Keum - Naie surface. Let $\hat{S} \rightarrow S$ be the \'etale
$\ZZ_2^2$-cover associated to $\langle
e_1, e_1', e_2, e_2' \rangle \triangleleft \Gamma$ and let
\begin{equation*}\label{stein}
\xymatrix{
\hat{S} \ar[r]^{\hat{\alpha}}\ar[dr]&E_1' \times E_2'\\ &Y \ar[u]_{\varphi}\\ }
\end{equation*} be the Stein factorization of the Albanese map of $\hat{S}$.

\noindent Then $\varphi$ has degree $2$ and $Y$ is a canonical model
of $\hat{S}$.

More precisely, $\varphi$  is a
double cover of $E_1' \times E_2'$
branched on a divisor of type $(4,4)$.

\end{proposition}

The fact that $S$ is homotopically equivalent to a Keum-Naie surface
immediately implies that the degree of
$\hat{\alpha}$ is equal to two.

The second assertion, i.e., that $Y$ has only canonical
singularities, follows instead from standard formulae
on double covers (cf. \cite{horikawa78}).

The last assertion follows from $K^2_{\hat{S}} = 16$ and $(\ZZ/2\ZZ)^2$-
invariance.

In fact, we conjecture a stronger statement to hold true:
\begin{conjecture}\label{conj1}
   Let $S$ be a minimal smooth projective surface such that
\begin{itemize}
   \item[i)] $K_S^2 = 4$,
\item[ii)] $\pi_1(S) \cong \Gamma$.
\end{itemize} Then $S$ is a Keum-Naie surface.
\end{conjecture}

We can prove
\begin{theorem}\label{Kample}
   Let $S$ be a minimal smooth projective surface such that
\begin{itemize}
   \item[i)] $K_S^2 = 4$,
\item[ii)] $\pi_1(S) \cong \Gamma$,
\item[iii)] there is a deformation of $S$ with ample canonical bundle.
\end{itemize} Then $S$ is a Keum-Naie surface.
\end{theorem}

We recall the following results:

\begin{theorem}[Severi's conjecture, \cite{pardini}]\label{sevconj}
   Let $S$ be a minimal smooth projective surface of maximal Albanese
dimension (i.e., the image of the
Albanese map is a surface), then $K_S^2 \geq 4 \chi(S)$.
\end{theorem}

M. Manetti proved Severi's inequality under the assumption that $K_S$
is ample, but he also gave a
description of the limit case $K_S^2 = 4 \chi(S)$, which will be
crucial for the above theorem \ref{Kample}.

\begin{theorem}[M. Manetti, \cite{manetti}]\label{SevMan}
   Let $S$ be a minimal smooth projective surface of maximal Albanese
dimension with $K_S$ ample then
$K_S^2 \geq 4 \chi(S)$, and equality holds if and only if $q(S) = 2$,
and the Albanese map $\alpha : S
\rightarrow \Alb(S)$ is a finite double cover.
\end{theorem}

\begin{proof}[Proof of theorem \ref{Kample}]
   We know that there is an \'etale $\ZZ_2^2$-cover $\hat{S}$ of $S$
with Albanese map $\hat{\alpha} :
\hat{S} \rightarrow E_1' \times E_2'$. Note that $K_{\hat{S}}^2 = 4
K_S^2 = 16$. By Severi's inequality, it
follows that $\chi(S) \leq 4$, but since $1 \leq \chi(S) =
\frac{1}{4} \chi(\hat{S})$, we have $\chi(S) = 4$.
Since $S$ deforms to a surface with $K_S$ ample, we can apply
Manetti's result and obtain that
$\hat{\alpha} : \hat{S} \rightarrow E_1' \times E_2'$ has degree $2$,
and we conclude as before.
\end{proof}

   It seems reasonable to conjecture (cf. \cite{manetti}) the
following, which would immediately imply our
conjecture \ref{conj1}.
\begin{conjecture}
   Let $S$ be a minimal smooth projective surface of maximal Albanese
dimension. Then $K_S^2 = 4 \chi(S)$ if
and only if $q(S) = 2$, and the Albanese map has degree $2$.
\end{conjecture}

During the preparation of the article \cite{keumnaie} the authors
realized that a completely similar
argument applies to {\em primary Burniat surfaces}.

We briefly recall
the construction of Burniat surfaces:
    for more details, and for the proof that Burniat surfaces are
exactly certain Inoue surfaces we refer to
\cite{burniat1}.

Burniat surfaces are minimal surfaces of general type with $K^2
=6,5,4,3,2$ and $p_g = 0$, which were
constructed in \cite{burniat} as singular bidouble covers (Galois
covers with group $\ZZ_2^2$) of the
projective plane branched on 9 lines.

\medskip
\noindent Let $P_1, P_2, P_3 \in \PP^2$ be three non collinear points
(which we assume to be the points
$(1:0:0)$, $(0:1:0)$ and $(0:0:1)$)  and let's denote by
$Y:=\hat{\PP}^2(P_1, P_2,P_3)$ the Del Pezzo surface of degree $6$,
blow up of $\PP^2$ in $P_1, P_2, P_3$.

     $Y$ is `the' smooth Del Pezzo surface of degree $6$, and it is the
closure of the graph of the rational map
$$\epsilon: \PP^2 \dashrightarrow \PP^1 \times \PP^1 \times \PP^1$$
such that $$\epsilon (y_1 : y_2:
y_3)  = ((y_2 : y_3), ( y_3: y_1), ( y_1: y_2)).$$

One sees immediately that $Y \subset \PP^1 \times \PP^1 \times \PP^1$
is the hypersurface of type
$(1,1,1)$:
$$ Y = \{(( x_1' : x_1), (  x_2':  x_2), ( x_3':  x_3)) \ | \ x_1 x_2
x_3 = x_1' x_2' x_3' \}.$$

We denote by $L$ the total transform of a general line in $\PP^2$, by
$E_i$ the exceptional curve lying over $P_i$,   and by $D_{i,1} $ the
unique effective divisor in $ |L - E_i-
E_{i+1}|$, i.e., the proper transform of the line $y_{i-1} = 0$, side
of the triangle joining the points $P_i,
P_{i+1}$.

Consider on $Y$,  for each   $  i \in  \ZZ_3 \cong \{1,2,3\}$,
    the following divisors
$$ D_i = D_{i,1} + D_{i,2} + D_{i,3} + E_{i+2} \in |3L - 3E_i -
E_{i+1}+E_{i+2}|,$$

where $D_{i,j} \in |L - E_i|, \ \rm{for} \ j = 2,3, \  D_{i,j} \neq
D_{i,1}$, is the proper transform of another
line through
$P_i$ and $D_{i,1} \in |L - E_i- E_{i+1}|$ is as above. Assume also
that all the corresponding lines in $\PP^2$
are distinct, so that $D : = \sum_i D_i$ is  a reduced divisor.

Note that, if we define  the divisor $\mathcal{L}_i : = 3L - 2
E_{i-1} - E_{i+1}$, then
$$D_{i-1} + D_{i+1} = 6L - 4 E_{i-1} - 2E_{i+1} \equiv 2
\mathcal{L}_i,$$ and we can consider (cf.  section 4, \cite{cat1} and 
\cite{sbc}) the associated
bidouble cover $X' \rightarrow Y$
branched on $D : = \sum_i D_i$ (but we take a different ordering of
the indices of the fibre coordinates
$u_i$, using the same choice as the one made in \cite{burniat1},
except that $X'$ was denoted by $X$).

We recall that this precisely means the following: let $D_i =
\divi(\delta_i)$, and let $u_i$ be
       a fibre coordinate of the geometric line bundle $\LL_{i+1}$,
whose sheaf of holomorphic sections is
$\hol_Y(\mathcal{L}_{i+1})$.

Then $X \subset \LL_1 \oplus \LL_2 \oplus
\LL_3$ is given by the equations:
$$ u_1u_2 = \delta_1 u_3, \ \ u_1^2 = \delta_3 \delta_1;
$$
$$ u_2u_3 = \delta_2 u_1, \ \ u_2^2 = \delta_1 \delta_2;
$$
$$ u_3 u_1 = \delta_3 u_2, \ \ u_3^2 = \delta_2 \delta_3.
$$

      From the birational point of view, as done by Burniat, we are
simply adjoining to the function field of
$\PP^2$ two square roots, namely $\sqrt  \frac{\Delta_1}{ \Delta_3}$ and
$\sqrt  \frac{\Delta_2}{ \Delta_3}$, where $\Delta_i$ is the cubic
polynomial in $\CC[x_0,x_1,x_2]$ whose
zero set has $D_i - E_{i+2}$ as strict transform.

This shows clearly that we have a Galois cover $X' \ra Y$ with group
$\ZZ_2^2$.

The equations above give a biregular model $X'$ which is nonsingular
exactly if the divisor $D$ does not have
points of multiplicity 3 (there cannot be points of higher
multiplicities!). These points give then quotient
singularities of type $\frac{1}{4} (1,1) $, i.e., isomorphic to the
quotient of $\CC^2$ by the action of $\ZZ_4$
sending
$ (u,v) \mapsto (iu, iv)$ (or, equivalently, the affine cone over
the 4-th Veronese embedding of $\PP^1$).

\begin{definition}
       A {\em primary Burniat surface} is a surface constructed as
above, and which is moreover smooth.
       It is then a minimal surface $S$ with $K_S$ ample, and with
$K_S^2 = 6$, $p_g(S) =q(S)= 0$.

        A  {\em secondary Burniat surface} is the minimal resolution of
a surface $X'$ constructed as above, and
which moreover has
$ 1 \leq m \leq 2$ singular points (necessarily of the type described above).
       Its minimal resolution is then a minimal surface $S$ with $K_S$
nef and big, and with $K_S^2 = 6-m$,
$p_g(S) =q(S)= 0$.

       A {\em tertiary (respectively, quaternary) Burniat surface} is
the minimal resolution of a surface $X'$
constructed as above, and which  moreover has $ m = 3 $ (respectively
$ m=4$) singular points (necessarily of the type described above).
       Its minimal resolution is then a minimal surface $S$ with $K_S$
nef and big, but not ample,  and with
$K_S^2 = 6-m$,
$p_g(S) =q(S)= 0$.

\end{definition}

\begin{remark} 1) We remark that for $K_S^2 =4$ there are two possible
types of configurations. The  one where
there are three collinear points of multiplicity at least 3 for the
plane curve formed by the 9 lines leads to a
Burniat surface $S$ which we call of  {\em nodal type}, and with
$K_S$ not ample, since the inverse image of the line joining the 3
collinear points is a (-2)-curve (a smooth
rational curve of self intersection $-2$).

     In the other cases with $K_S^2 =4, 5, 6$, $K_S$ is instead ample.

2) In the nodal case,  if we  blow up  the two $(1,1,1)$ points of
$D$, we obtain a weak Del Pezzo surface $\tilde{Y}$, since it
contains a (-2)-curve. Its anticanonical model
$Y'$ has a node (an
$A_1$-singularity, corresponding to the contraction of the
(-2)-curve). In the non nodal case, we obtain a
smooth Del Pezzo surface $\tilde{Y} = Y'$ of degree $4$.
\end{remark}

With similar methods as in \cite{keumnaie} (cf. \cite{burniat1}) the
first two authors  proved

\begin{theorem} The  subset of the Gieseker moduli space corresponding
to primary Burniat surfaces is an
irreducible connected component, normal, rational and of dimension
four. More generally, any
surface homotopically equivalent to a primary Burniat surface is
indeed a primary Burniat surface.
\end{theorem}

\begin{remark}
   The assertion that the moduli space
corresponding to primary Burniat
surfaces is rational  needs indeed a further argument, which is
carried out in \cite{burniat1}.
\end{remark}


\end{document}